\documentclass[aap]{imsart}
\RequirePackage{amsthm,amsmath,amsfonts,amssymb}
\RequirePackage[authoryear]{natbib}
\RequirePackage[colorlinks,citecolor=blue,urlcolor=blue]{hyperref}
\RequirePackage{graphicx}

\startlocaldefs
\theoremstyle{plain}
\newtheorem{thm}{Theorem}[section]
\newtheorem{prop}[thm]{Proposition}

\newtheorem{ass}{Assumption}
\newtheorem{lem}[thm]{Lemma}
\newtheorem{rem}[thm]{Remark}

\newcommand\auxY{\widetilde{X}}
\newcommand\indiq{{\mathbf 1}}

\newcommand\esp[1]{\mathbb{E}\left[#1\right]}
\newcommand\espc[2]{\mathbb{E}\left[\left.#1\right|#2\right]}
\newcommand\N{\n}
\newcommand\n{\mathbb{N}}
\renewcommand\r{\mathbb{R}}
\newcommand\R{\r}
\newcommand\E{\mathbb{E}}
\newcommand\F{\mathcal{F}}

\newcommand\xd{\bar X}
\newcommand\cE{\mathcal{E}}
\newcommand\cR{\mathcal{R}}
\newcommand\xnone{X^{N, 1 }}
\newcommand\limY{\bar{X}}
\newcommand\W{\mathcal{W}}
\newcommand\C{\mathcal{C}}
\endlocaldefs

\begin{document}

\begin{frontmatter}

\title{Strong error bounds for the convergence to its mean field limit  for systems of interacting neurons in a diffusive scaling}

\runtitle{Strong error bounds in conditional propagation of chaos}

\begin{aug}
\author[A]{\fnms{Xavier}~\snm{Erny}\ead[label=e1]{xavier.erny@polytechnique.edu}},
\author[B]{\fnms{Eva}~\snm{L\"ocherbach}\ead[label=e2]{eva.locherbach@univ-paris1.fr}}
\and
\author[C]{\fnms{Dasha}~\snm{Loukianova}\ead[label=e3]{dasha.loukianova@univ-evry.fr}}

\address[A]{CMAP, CNRS, \'Ecole polytechnique, Institut Polytechnique de Paris, 91120 Palaiseau, France \printead[presep={,\ }]{e1}}

\address[B]{Statistique, Analyse et Mod\'elisation Multidisciplinaire, Universit\'e Paris 1 Panth\'eon-Sorbonne, EA 4543 \printead[presep={,\ }]{e2}}

\address[C]{Laboratoire de Math\'ematiques et Mod\'elisation d'\'Evry,  Universit\'e
  d'\'Evry Val d'Essonne, UMR CNRS 8071 \printead[presep={,\ }]{e3}}
\end{aug}

\begin{abstract}
We consider the stochastic system of interacting neurons introduced in
 \cite{de_masi_hydrodynamic_2015} and in \cite{fournier_toy_2016} and then further studied in \cite{ELL_2020}  in a diffusive scaling. The system consists of $N$ neurons, each spiking randomly with rate depending on its membrane potential. 
At its spiking time, the potential of the spiking neuron  is reset to $0$ and all other neurons
receive an additional amount of potential which is a centred random variable of order $ 1 / \sqrt{N}.$ In between successive spikes, each neuron's potential follows a 
deterministic flow.  In our previous article \cite{ELL_2020}  we proved the convergence of the system, as $N \to \infty$, to a limit nonlinear
jumping stochastic differential 
equation.  In the present article we complete this study by establishing a strong convergence result, stated with respect to an appropriate distance, with an explicit rate of convergence. The main technical ingredient of our proof is the coupling introduced in \cite{komlos_approximation_1976} of the point process representing the small jumps of the particle system with the limit Brownian motion.\\
\end{abstract}


\end{frontmatter}

\section{Introduction}
In the present paper we continue the study started in \cite{ELL_2020} and establish strong error bounds for a mean field limit of systems of interacting neurons in a diffusive scaling. More precisely, we are interested in  the large population limit of the Markov process $X^N = (X^N_t)_{t \geq 0 }, $ $X^N_t = (X^{N, 1 }_t, \ldots , X^{N, N}_t )$ which takes values in $\r^N$  and has generator $A^N$ given by 
\begin{equation}\label{eq:dynintro}
A^N  \varphi ( x) = - \alpha \sum_{i=1}^N \partial_{x^i} \varphi (x) x^i + \sum_{i=1}^N f (x^i) \int_\R \nu ( du ) \left( \varphi ( x - x^i e_i + \sum_{j\neq i } \frac{u}{\sqrt{N}} e_j ) - \varphi ( x) \right) ,
\end{equation}
for any smooth test function $ \varphi : \R^N \to \R .$ In the above formula,  $ x= (x^1, \ldots, x^N) \in \r^N$ is the vector of membrane potential values of the $N$ neurons and  $ e_j $ denotes the $j-$th unit vector in $ \R^N.$ Moreover, $ \alpha > 0 $ is a fixed parameter, $ f : \R \to \R_+$ a Lipschitz continuous rate function and $ \nu $ a centred probability measure on $\R$ having a second moment. System \eqref{eq:dynintro} is a version of the model of interacting neurons considered in \cite{de_masi_hydrodynamic_2015}, inspired by  \cite{galves_infinite_2013}, and then further studied, among others, in  \cite{fournier_toy_2016}, \cite{touboul} and \cite{cormier}. 
The system consists of $N$ interacting and spiking neurons represented by their membrane potential values. Spiking occurs randomly following a point process of rate $f (x) $ for any neuron of which the membrane potential equals $x.$  Each time a neuron emits a spike, the potentials of all other neurons receive an additional amount of potential. In \cite{de_masi_hydrodynamic_2015}, \cite{fournier_toy_2016}, \cite{touboul} and \cite{cormier} this amount is of order $N^{-1}, $ leading to classical mean field limits as $ N \to \infty .$ On the contrary to this, here we continue the work started in \cite{ELL_2020} where we consider a {\it diffusive scaling} in which all neurons $j$ receive the same random quantity  $U/\sqrt{N}$ at spike times $t$ of neuron $i, i \neq j .$ The random variable $U$ is distributed according to the fixed probability measure $ \nu, $ it is chosen independently of anything else at each spike time and it is centred modeling the fact that the synaptic weights are balanced.  Moreover, right after its spike, the potential of the spiking neuron~$i$ is reset to~0, interpreted as resting potential. Finally, in between successive spikes, each neuron has a loss of potential of rate~$\alpha$.

In \cite{ELL_2020} we have established a weak convergence result, namely, we have proved that for all $ K > 0,$ the joint law $ {\mathcal L }(X^{N, 1, }, \ldots , X^{N, K} )$ converges weakly to a limit law ${\mathcal L} ( \limY^1, \ldots, \limY^K) $ in $D (\R_+, \R )^K .$ Here, the limit process $ ( \limY^i)_{i \geq 1 } $ is an infinitely exchangeable system given by 
\begin{equation}
\label{eq:dynlimintro}
\limY^i_t = \limY^i_0 -\alpha   \int_0^t \limY^i_s   ds - \int_0^t \limY^i_{s- } d\bar{Z}^{i}_s + \sigma \int_0^t \sqrt{ \E ( f ( \bar X^i_s ) | {\W_s  }) } d W_s ,\; t \geq 0,   i \in \N,
\end{equation}
$\sigma^2 = \E (U^2 ) .$ 
In the above system, each counting process $ \bar{Z}^i $ has intensity $ t \mapsto f ( \limY^i_{t- }),$ $ W$ is a standard one dimensional Brownian motion, created by the central limit theorem, representing a source of {\it common noise}, and $ \W_s = \sigma \{ W_t , t \le s \} .$ 

To prove the convergence in distribution of the finite system to the above limit system, in \cite{ELL_2020} we made use of a new martingale problem, and we used the exchangeability both of the finite and the limit system. In particular, we proved there that the {\it conditional propagation of chaos} property holds, that is, in the limit system, neurons are conditionally independent, if we condition on the source of common noise. And this source of common noise is the presence of the Brownian motion $W$ which appears as a consequence of the central limit theorem. 

In the present paper, we complete the above study and prove the strong convergence of a given neuron, say $ X^{N, 1}, $ to its corresponding limit quantity $ \bar X^1.$ To do so, we couple the point processes underlying the evolution of~\eqref{eq:dynintro} with the Brownian motion appearing in the limit equation ~\eqref{eq:dynlimintro} using ideas that go back  to \cite{kurtz78}. This coupling is based on a corollary of the KMT inequality (see Theorem~1 of \cite{komlos_approximation_1976}). 

The idea of using the KMT coupling in this context is not new. It has been successfully applied to interacting particle systems in the scaling $ 1/ N, $ corresponding to the scale of the law of large numbers, in \cite{kurtz78}, an approach that has been improved in the recent paper  \cite{adrien}. To the best of our knowledge, to use this coupling in the diffusive scaling is however entirely new.

Let us now describe more precisely the strategy of our coupling. Before doing so, we precisely define the processes we are interested in and recall important results obtained before that we rely on in the sequel. 

\subsection{Notation}
In what follows, $C$ denotes an arbitrary positive constant whose values can change from line to line in an equation. We write $C_t$  if the constant depends on time $t .$ We write $C(\r_+,\r)$ for the space of continuous functions from $\r_+$ to $\r ,$ $C^k_b(\R , \R_+ )$ for the space of all bounded continuous functions from $ \R $ to $\R_+$ being $k$ times differentiable with bounded derivatives up to order $k,$ and finally $ \r^{\N} $ for the space of all functions from $ \N $ to $ \r.$ 

Finally, to ease notation, we shall write $ \bar f (x) = \sum_{i=1}^N f( x^i ), $ for any $x = (x^1, \ldots, x^N) \in \R^N.$ We will also frequently use the notation $ [s] = \max \{ n \in \N : n \le s \} $ for the integer part of a positive real number $s \in \R_+.$

\subsection{A representation by means of Poisson random measures}
Let $\nu$ and $\nu_0$ be two probability measures on $\r, $ where $\nu$ is centered. A first ingredient of our coupling is to construct the two systems using the same underlying jump noise. To do so, it is convenient to represent 
the evolution of both systems as the solution of a stochastic differential equation driven by Poisson random measures. For that sake, we consider an i.i.d. family of  Poisson measures $(\pi^i(ds, dz, du))_{i \geq 1 }$ on $\r_+  \times \r_+ \times \r $ having intensity measure $ds dz \nu (du)$, as well as an i.i.d. family $(X^{i}_0)_{i \geq 1 }$ of $\r $-valued random variables, independent of the Poisson measures, distributed according to  $ \nu_0 .$   Denoting $\bar \pi^i (ds, dz) = \pi^i ( ds, dz, \R) ,$ the projection of $\pi^i $ on its first two coordinates, 
we may represent each neuron's potential within the finite system as
\begin{eqnarray}\label{eq:dyn}
X^{N, i}_t &=&  X^{i}_0  - \alpha  \int_0^t  X^{N, i}_s   ds -  \int_{[0,t]\times\r_+ } X^{N, i}_{s-}  \indiq_{ \{ z \le  f ( X^{N, i}_{s-}) \}} \bar \pi^i (ds,dz) \\
&&+ \frac{1}{\sqrt{N}}\sum_{ j \neq i } \int_{[0,t]\times\r_+\times\r}u \indiq_{ \{ z \le  f ( X^{N, j}_{s-}) \}} \pi^j (ds,dz,du) ,\nonumber 
\end{eqnarray} 
and 
the associated limit system  $\left(\limY^i\right)_{i\geq 1}$  as
\begin{equation}\label{eq:dynlimit1}
\limY^i_t =  X^i_0 -  \int_0^t  \alpha \limY^i_s  ds -\int_{[0,t]\times\r_+ }  \limY^i_{s- } \indiq_{ \{ z \le  f ( \limY^i_{s-}) \}} \bar \pi^i  (ds,dz) 
 +  \int_0^t \sqrt{\espc{f\left(\limY^i_s\right)}{\W_s  }} d W_s ,
\end{equation}
starting from the same initial positions $ X^i_0, i \geq 1, $ and driven by the same $ \bar \pi^i, i \geq 1, $  as the finite system. 
In the above equation, $(W_s)_{s\geq 0}$ is a standard one-dimensional Brownian motion which is independent of the Poisson random measures $ (\bar \pi^i (ds, dz))_{ i \geq 1 }  $ 
and of $  X^i_0  , i \geq 1.$ Moreover, $ \W_s = \sigma \{ W_t , t \le s \} .$  

The above system \eqref{eq:dynlimit1} is an infinite exchangeable system, and $ W$ is the only source of common noise. In particular, for each fixed $ K, $ $\limY^1,\ldots, \limY^K$ are i.i.d. conditionally to $\W .$
Under appropriate assumptions on $f, \nu$ and $\nu_0$ (Assumptions \ref{ass:1} and \ref{ass:2} of Section \eqref{sec:2} below)  the existence and the well-posedness of the above systems are shown in   \cite{ELL_2020}. 
%

\subsection{The coupling}
Equations~\eqref{eq:dyn} and~\eqref{eq:dynlimit1} define already partially a coupling of the finite and the limit system, since we use the same underlying $ \bar \pi^i $ to produce the reset jumps of the spiking neurons. The next step of our coupling is to couple the Brownian motion $ W$ with the small jumps $ u/ \sqrt{N} $ appearing in the second line of \eqref{eq:dyn}. To do so, we  inspect the process 
$$\xi^N_t =\sum_{ j = 1}^N \int_{[0,t]\times\r_+\times\r}u \indiq_{ \{ z \le  f ( X^{N, j}_{s-}) \}} \pi^j (ds,dz,du)$$
appearing in \eqref{eq:dyn} above. Notice that, since the variables $u$ are centered (recall that $\int u d \nu (u) = 0$), the above expression actually defines a purely discontinuous martingale. 

Up to an error of order $ N^{- 1/2}, $ $N^{-1/2} \xi^N_t$  represents the contribution of the small jumps to the membrane potential of any fixed neuron, up to time $t.$

\subsubsection{A first naive approach}\label{sec:kurtz}
To fix ideas, we start describing a first naive approach. It turns out that this naive approach does not work in our frame - but it will give us a hint about the strategy to adopt.

Introducing 
$$ A^N_t = \int_0^t \sum_{j=1}^N f ( X^{N, j }_s) ds = \int_0^t \bar f( X^N_s) ds  ,$$ 
the time changed process $ Z:= \xi^N \circ (A^N)^{-1} $ is a centered compound Poisson process having rate $1$ and jump distribution $\nu$ (Prop. 14.6.III of \cite{DaleyII}).

A celebrated result due to \cite{komlos_approximation_1975, komlos_approximation_1976} shows that it is possible -- under the condition that the jump size distribution $ \nu$ admits exponential moments --  to construct  $Z$ together with a standard one-dimensional Brownian motion $ B, $  on the same probability space,  such that we have the pathwise control
$$ \sup_{t \geq 0} \frac{ |Z^{}_t -  B^{} _t|}{ \ln (t \vee 2) } \le E < \infty , $$
where $E$ is a random variable admitting exponential moments (Corollary 5.5 in Chapter 7 of \cite{ethier_markov_2005}). Therefore, we may rewrite the evolution of any (say the first) neuron as
\begin{eqnarray}\label{eq:firststep}
 X^{N, 1}_t &=& X^{1}_0  - \alpha  \int_0^t  X^{N, 1}_s   ds -  \int_{[0,t]\times\r_+\times\r} X^{N, 1}_{s-}  \indiq_{ \{ z \le  f ( X^{N, 1}_{s-}) \}} \pi^1 (ds,dz,du) + \frac{1}{\sqrt{N}}Z_{A_t^N} \nonumber \\
&=& X^{1}_0  - \alpha  \int_0^t  X^{N, 1}_s   ds -  \int_{[0,t]\times\r_+\times\r} X^{N, 1}_{s-}  \indiq_{ \{ z \le  f ( X^{N, 1}_{s-}) \}} \pi^1 (ds,dz,du) \nonumber  \\
&& \quad \quad  \quad \quad \quad  \quad \quad   \quad \quad \quad \quad  \quad \quad   \quad \quad \quad \quad  \quad \quad  \quad \quad  \quad \quad   \quad \quad  \quad \quad+ \frac{1}{\sqrt{N}}B_{A_t^N} + K_t ,
 \end{eqnarray}
 where $ B_{A_t^N} $ is the time changed Brownian motion, using the integrated jump intensity as time change, and where the error term is such that 
 \begin{equation}\label{eq:kt}
 | K_t | \le  \frac{\ln ( t N \|f\|_\infty)  E}{\sqrt{N}}  .
\end{equation} 
  
A natural idea is to compare \eqref{eq:firststep} to the evolution of an auxiliary particle system where the error term is absent.  
This auxiliary particle system, interpolating between the finite system $ X^N $ and the limit system $ \bar X ,$ should therefore be defined by 
\begin{equation}\label{eq:auxnaive}
 \tilde X^{N, 1}_t = X^{1}_0  - \alpha  \int_0^t  \tilde X^{N, 1}_s   ds -  \int_{[0,t]\times\r_+\times\r} \tilde X^{N, 1}_{s-}  \indiq_{ \{ z \le  f ( \tilde X^{N, 1}_{s-}) \}} \pi^1 (ds,dz,du) 
 + \frac{1}{\sqrt{N}}B_{\tilde A_t^N} ,
\end{equation}
completed with an analogous evolution for the other particles $ \tilde X^{N, j} , j \geq 2,$ such that the evolution is exchangeable, and using the same Brownian motion $ B$ and the same Poisson random measures $ \pi^i $ as in \eqref{eq:firststep}. In the above evolution, $ \tilde A^N_t =  \int_0^t \sum_{j=1}^N f ( \tilde X^{N, j }_s) ds .$ To compare the two systems, we need to compare the time changed Brownian motions, that is, to produce a control of 
$$  \frac{1}{\sqrt{N}}B_{A_t^N} -  \frac{1}{\sqrt{N}}B_{\tilde A_t^N},$$
relying on the continuity properties of the Brownian motion. Since Brownian motion is almost H\"older continuous of order $ 1/2$, up to an error term which is logarithmic, 
$$ \left|  \frac{1}{\sqrt{N}}B_{A_t^N} -  \frac{1}{\sqrt{N}}B_{\tilde A_t^N} \right| \le R \sqrt{ \frac{|A_t^N - \tilde A_t^N |}{N} } ,$$
where $R$ is the logarithmic error term. Disregarding this error term, we therefore end up with a control
$$ | X^{N, 1}_t -  \tilde X^{N, 1}_t  | \le R  \sqrt{ \frac{|A_t^N - \tilde A_t^N |}{N} } + \mbox{ other terms. }$$ 
Obviously, the presence of the square root on the rhs of the above inequality makes it impossible to use Gronwall's lemma and to take expectations -- this is however what one needs to obtain a precise control of the error within this approximation. 
We invite the interested reader to \cite{kurtz78}, where clever tricks are applied that overcome the above difficulty in the $1/N-$scaling corresponding to the law of large numbers and the study of the associated fluctuations. These tricks do however not apply in the present $1/\sqrt{N}-$scaling, see Section \ref{sec:appendixkurtz} below for more precise arguments. Another problem of this naive approach is that necessarily $B$ and the Poisson random measures $(\pi^i)_{i \geq 1}  $ are dependent -- while we wish to construct the limit system such that Brownian motion and Poisson random measures are defined with respect to the same filtration -- and hence necessarily independent (see e.g. Theorem II.6.3 of \cite{ikeda_stochastic_1989}). 

In what follows we show how to improve the above approach such that we will be able to avoid to use the modulus of continuity of Brownian motion and to produce a Brownian motion that is independent of the Poisson random measures.

\subsubsection{Freezing positions: Time-discretization and a piecewise KMT coupling}
To overcome the above difficulty, we freeze the coefficients over short time intervals such that the martingale $S^N$ of small jumps with frozen coefficients is indeed equal to a centered compound Poisson process and such that we can use a suitable coupling on each such time interval without being bothered by time changes. 
More precisely we decompose time into small intervals of length $ \delta $ where $\delta = \delta (N)  < 1$ has to be chosen in a convenient way. We shall use the KMT coupling on each such interval and then concatenate all these couplings pretty much in the same spirit as in \cite{adrien}. 

More precisely, we decompose for any $ k \geq 0$ the evolution of a given, say the first, neuron according to 
\begin{eqnarray}\label{eq:Euler1}
\xnone_{(k+1) \delta } &=& \xnone_{k\delta} - \alpha \int_{k \delta}^{(k+1)\delta} \xnone_{ s }ds 
- \int_{(k\delta, (k+1) \delta]  \times \R_+ }  \xnone_{s- } \indiq_{\{ z \le f ( \xnone_{ s- }) \}}\bar \pi^1 (ds, dz) \\
&&+ \frac{1}{\sqrt{N}}\sum_{j=1}^N  \int_{ (k\delta, (k+1) \delta]\times \R_+ \times \R  }  u    \indiq_{\{ z \le f ( X^{N, j }_{k\delta}) \}}  \pi^j ( ds, dz, du) + R^k_\delta , \nonumber
\end{eqnarray}
where $R^k_\delta$ is an error term.

If we condition on the values of the process at time $ k \delta  $  in \eqref{eq:Euler1}, say $ (x^1, \ldots, x^N) ,$ then the centered process of small jumps 
\begin{equation}\label{eq:centeredjumps}
t \mapsto \sum_{j=1}^N  \int_{ (k\delta, t ] \times \R_+ \times \R  }  u    \indiq_{\{ z \le f ( X^{N, j }_{k \delta}) \}}  \pi^j ( ds, dz, du) ,  k \delta < t \le  (k+1) \delta  ,
\end{equation}
now has constant 
jump rate given by $ \bar f (x) = \sum_{j=1}^N f ( x^j ) . $ The scaling property of Brownian motion  enables us to directly couple this process of small jumps, using the KMT construction, to a centered Brownian motion having variance $ \bar f (x),  $ that is, to  
$$ \sqrt{  \bar f(x) }  W^{N, k }_{t} ,$$ where $ W^{N, k } $ is a one dimensional standard Brownian motion,  independent of $ {\mathcal F}_{ k \delta } ,$ the past before time $ k \delta ,$ and of the projected Poisson random measures $ (\bar \pi^i )_{ i \geq 1}.$

An important point of our argument is that we use the time discretization to decouple the acceptance/rejection procedure -- which is represented by the terms $\indiq_{\{ z \le f ( X^{N, j }_{k \delta}) \}} $ in \eqref{eq:centeredjumps}, depending only on the frozen position of the neurons at time $ k \delta $ and performed using the projected measures $ \bar \pi^j , 1 \le j \le N $ -- from the random jump heights $ U$ representing the third coordinates of the atoms of  $\pi^j , 1 \le j \le N.$ We construct the Brownian motion $ W^{N, k }$ using only these centered random variables $U .$  In this way, $ W^{N, k } $ is indeed independent of the $ (\bar \pi^i )_{ i \geq 1}$ (but not of the original $ (\pi^i)_{i \geq 1}$).  

We perform such a coupling on each time interval $ (k \delta, (k+1 ) \delta ] $ and obtain independent Brownian motions $ W^{N, k } $ for each of them, together with error terms $K^{[k]} $ for each such interval. We then glue all these pieces of Brownian motion together to obtain a global Brownian motion $ W^N$ defined on $ \R_+$ which we use to construct the limit system \eqref{eq:dynlimit1}, driven by this particular Brownian motion $ W^N . $ 

Notice that in this way, the coupling is tailored to work for fixed size $N;$ changing the size of the system would lead to the construction of a different Brownian motion and hence of a different coupling.

The total error that contributes to the distance between $ \xnone_t$ and $ \bar X_t^1 $ is then given as the consequence of the following contributions. 
\begin{itemize}
\item 
The contribution of the error terms $ R_\delta^{  k }  ,  k \le t/ \delta , $ in \eqref{eq:Euler1}. Writing $ \tau ( s) = k \delta $ if $ k \delta \le s < (k+1) \delta, k \geq 0,$ this contribution is of order $ C_t \delta^{1/4},$ due to the fact that 
$$ \E \left( \left|  \frac{1}{\sqrt{N}} \sum_{j=1}^N  \int_{[0, t ] \times \R_+ \times \R  }  u   (  \indiq_{\{ z \le f ( X^{N, j }_{ s-}) \}}  - \indiq_{\{ z \le f ( X^{N, j }_{\tau ( s-)}) \}})  \pi^j ( ds, dz, du) \right| \right)  \le C_t \delta^{1/4} ,$$
which follows from  the $L^2-$isometry for stochastic integrals with respect to compensated Poisson random measures and the fact that $ \E | X^{N, j}_t   -  X^{N, j }_{\tau ( t) } | \le C_t \sqrt{\delta}$ (see Appendix). 
\item
The contribution coming from all error terms $ K^{[k]}, $ due to the coupling attempts on each time interval $ (k \delta, (k+1) \delta ].$ We upper bound this contribution by the sum of all contributions (we have $[t/\delta]$ such contributions) which is given by 
$$  C_t \frac{\ln   N  }{\sqrt{N}} [\frac{t}{ \delta}] .$$
\item
The contribution coming from the fact that we approximate the limit system \eqref{eq:dynlimit1} by its mean field version where the stochastic volatility $ \E ( f( \bar X^i_t) | {\W_t  }) $ is replaced by its mean-field version $ \frac1N \sum_{i=1}^N f ( \bar X^{i }_t) .$ This is a variance term of order $ C_t / \sqrt{N}. $ 
\end{itemize}
Choosing $ \delta = \delta (N) $ such that all these error terms are balanced, we end up with a strong error of the form $ C_t   (\ln N)^{1/5} N^{ - 1/10},$ and this is the content of our main Theorem \ref{strongconvergence}  stated in Section \ref{sec:2} below.

\subsection{A reader's guide through this article}
Most of the technical proofs of our paper rely on the same ideas. Therefore we have given the explicit details of one of these proofs, the proof of Theorem \ref{th:auxiliary}
in Subsection \ref{sec:35}, and the proof of Proposition \ref{prop:2} given in Section \ref{sec:41}. In most of the other proofs, these details are then skipped.

\section{Main results}\label{sec:2}
In the sequel, $ \nu $ and $ \nu_0$  denote probability measures on $(\r,\mathcal{B}(\r))$ satisfying 
\begin{ass}\label{ass:1}
\label{control}
$\int_\r xd\nu(x)=0,$ $\int_\r x^2d\nu(x)= 1,$ and $\int_\r x^2d\nu_0(x)<+\infty.$
\end{ass}
Recall that we consider a family of i.i.d. Poisson measures $(\pi^i(ds, dz, du))_{i \geq 1 }$ on $\r_+  \times \r_+ \times \r $ having intensity measure $ds dz \nu (du)$ and an i.i.d. family $(X^{i}_0)_{i \geq 1 }$ of $\r $-valued random variables independent of the Poisson measures, distributed according to  $ \nu_0 .$ In what follows we  use the associated canonical filtration  
\begin{equation}\label{eq:filtration}
 {\mathcal F}_t = \sigma \{  \pi^i  ( [0, s ] \times A ), s \le  t , A \in {\mathcal B} ( \R_+ \times \R ) , i \geq 1 \} \vee \sigma \{ X^{ i }_0, i \geq 1 \}, \; t \geq 0.
\end{equation}  
Recall also that the projected Poisson random measures are defined by 
\begin{equation}\label{eq:projpi}
 \bar \pi^i (ds, dz) = \pi^i ( ds, dz, \R) , 
\end{equation} 
having intensity $ds  dz .$ For any $ N \in\N, $ the finite system $ (X^{N, i}_{t}) ,  t\geq 0, \;  1 \le i \le N , $ is given by
\begin{eqnarray}
\label{eq:dynbis}
X^{N, i}_t &= & X^{i}_0  - \alpha  \int_0^t  X^{N, i}_s   ds -  \int_{[0,t]\times\r_+ } X^{N, i}_{s-}  \indiq_{ \{ z \le  f ( X^{N, i}_{s-}) \}} \bar \pi^i (ds,dz) \\
&&+ \frac{1}{\sqrt{N}}\sum_{ j \neq i } \int_{[0,t]\times\r_+\times\r}u \indiq_{ \{ z \le  f ( X^{N, j}_{s-}) \}} \pi^j (ds,dz,du) ,\nonumber
\end{eqnarray} 
and its associated limit system  $\left(\limY_t^i\right), t \geq 0,  {i\geq 1},$ by 
\begin{equation}
\label{eq:dynlimit1bis}
\limY^i_t =  X^i_0 -  \int_0^t  \alpha \limY^i_s  ds -\int_{[0,t]\times\r_+ }  \limY^i_{s- } \indiq_{ \{ z \le  f ( \limY^i_{s-}) \}} \bar \pi^i  (ds,dz) 
 +  \int_0^t \sqrt{\espc{f\left(\limY^i_s\right)}{\W_s  }} d W_s.
\end{equation}
In the above equation, $(W_s)_{s\geq 0}$ is a standard one-dimensional Brownian motion which is independent of the Poisson random measures $ (\bar \pi^i (ds, dz))_{ i \geq 1 }  $ 
and of $  X^i_0  , i \geq 1.$ Moreover, $ \W_s = \sigma \{ W_t , t \le s \} .$ To grant existence and the well-posedness of the above systems, we impose the following assumption.
\begin{ass}\label{ass:2}
1. $f$ is bounded and lower bounded such that $ \inf f > 0 .$ \\
2. Moreover, $f\in C_b^1(\r,\r_+)$ with
\begin{equation}\label{eq:eps}
|f'(x)|\leq \frac{C}{(1+|x|)^{1+\epsilon}}
\end{equation}
for all $x\in\r,$ where $C$ and $\epsilon$ are some positive constants. 
\end{ass}   
Notice that this assumption implies in particular that $f$ is Lipschitz. Under Assumption \ref{ass:2}, both systems \eqref{eq:dyn} and \eqref{eq:dynlimit1} are well-posed and admit a unique strong solution (Theorem~IV.9.1 of \cite{ikeda_stochastic_1989} for the finite neuron system and Theorem 2.6 of \cite{ELL_2020} for the limit system). 

Our result of strong convergence is stated in terms of a convenient distance that has already been introduced in \cite{ELL_2020}. This distance is used to deal with the big jump terms appearing both in the finite and in the limit system. 
More precisely, under Assumption \ref{ass:2} we may introduce the function 
$$ a (x) = \int_{-\infty}^x \frac{dy}{(1+\psi(y))^{1+\epsilon}},$$
where $\psi$ is any smooth non-negative function satisfying $\psi(y)=|y|$ for $|y|\geq 1, $ and where $ \epsilon$ is the bound appearing in the control on $ f' $ in Assumption \ref{ass:2}. The function $a (\cdot) $  belongs to  $ C^3_b(\R , \R_+ )$ and is strictly increasing. 
By construction, we have that  $ |f( x) - f (y ) | \le C | a ( x) - a(y) |, $ since $ |f' ( t) |  \le C | a' (t) | $ for all $ t \in \R.$ Moreover, one easily verifies that 
$$ |a''' (t) | + |a'' ( t) | + | t a '' (t) | \le C | a' (t) | , $$ 
for all $t \in \R, $ such that  for all $ x, y \in \R,$ 
\begin{equation}\label{eq:a}
 |a'' ( x) - a'' (y) | +  |a'(x) - a' (y) | + |x a'(x) - ya'(y) | + |f(x)-f(y)| \le C | a(x)-a(y) |,
\end{equation} 
where $C$ is some fixed constant.

To be able to perform our coupling, in addition to Assumption \ref{ass:2}, we also need to impose the following assumption on the jump size distribution.

\begin{ass}\label{ass:exp}
We assume additionally that $ \int_\R e^{ a x } \nu (dx) < \infty $ for all $ |a| \le a_0 $ for some $ a_0 >0.$ 
\end{ass}

In what follows, $ (\Omega, {\mathcal A}, {\mathbf P}) $ denotes a probability space where are defined the Poisson random measures $ (\pi^i )_{ i \geq 1}$ and the  initial positions $X^{i}_0, i \geq 1.$ 

\begin{thm}
\label{strongconvergence}
Grant Assumptions~\ref{control}, \ref{ass:2} and \ref{ass:exp} and fix $N\in\n^*.$ Let $a:\r\to\R_+$ satisfy \eqref {eq:a}.  Then it is possible to construct on an extension of $ (\Omega, {\mathcal A}, {\mathbf P}) $  a one-dimensional standard Brownian motion $W^N$ which is independent of  the initial positions $X^{i}_0, i \geq 1,$ and of $ (\bar \pi^i)_{i \geq 1} , i \geq 1 , $
such that  the following holds. Denoting $ (\bar X^{N, i })_{ i \geq 1 } $ the strong solution of \eqref{eq:dynlimit1} driven by $ W^N$ and by $ (\bar \pi^i )_{ i \geq 1},$ we have for every $t>0,$ and for any $ 1 \le j \le N, $ 
\begin{equation}
\label{strongbound}
 \E ( \sup_{ 0 \le s \le t } | a(X^{N,j}_s) -a(\limY^{N,j}_s) | ) \leq C_t  (\ln N)^{1/5} N^{ - 1/10} . 
\end{equation}
\end{thm}

\begin{rem}
The above result is stated with respect to the distance $ a ( \cdot ) $ in order to be able to deal with the big jumps since 
$$ \int_{\R} | x \indiq_{\{ z \le f (x) \}} - y \indiq_{\{z \le f (y) \}} | dz $$
is (in general) not of order $ C | x - y | .$ 
Imposing the existence of exponential moments for $ \nu_0, $ it is also possible to use the famous $x \ln x $ extension of the Gronwall lemma (Osgood's lemma) and to obtain upper bounds directly for $ \E ( \sup_{ 0 \le s \le t } | X^{N,j}_s - \limY^{N,j}_s| ) .$ The prize to pay is that in this case the rate of convergence will be of order 
$$ N^{  - \frac12 e^{ - C_t}} ,$$ 
without explicit expression of the constant $C_t,$ yielding a rate of convergence which may be worse than $ N^{ - 1/10} $ for large $t$ whenever $C_t$ grows with $t$.  
\end{rem}

To prove \eqref{strongbound}, by exchangeability, it is sufficient to concentrate only on the first neuron, and the proof will be given for $j=1.$
Using Ito's formula,
\begin{eqnarray}\label{eq:tobecited}
a(\xnone_{t }) &=& a(\xnone_{0}) 
- \alpha \int_{0}^{t} a' (\xnone_{ s }) \xnone_{ s }  ds\\
&& 
+ \int_{[0, t ]   \times \R_+  } [ a(0) - a( \xnone_{s- })]  \indiq_{\{ z \le f ( \xnone_{ s-}) \}}\bar \pi^{1} (ds, dz)+ M_t^N ,\nonumber
\end{eqnarray}
where 
\begin{equation}\label{eq:stn}
M_t^N:=\sum_{j\neq 1}^N\int_{[0,t ]  \times \R_+ \times \R }(a(X^{N,1}_{s-}+u/\sqrt N)-a(X^{N,1}_{s-})) \indiq_{\{ z \le f ( X^{N, j }_{ s-}) \}} \pi^{ j } (ds, dz,du).
\end{equation}
This last term represents the contribution of small jumps (of order $1/\sqrt N$) to the dynamic of $X_t^{N,1}.$  
When $N\to\infty, $ by the central limit theorem, the small jumps create a Brownian motion.  The proof of Theorem \ref {strongconvergence} is based on the coupling of  $M_t^N$ with an Ito integral directed by this Brownian motion. Recall that we write $ \bar f( x) = \sum_{i=1}^N f ( x^i). $ We have the following representation. 

\begin{thm}\label{th:auxiliary}
Grant Assumptions~\ref{control}, \ref{ass:2} and \ref{ass:exp} and fix $N\in\n^*$ and a time horizon $t> 0.$ Then it is possible to construct on an extension of $ (\Omega, {\mathcal A}, {\mathbf P}) $  a one-dimensional standard Brownian motion $W^N$ which is independent of  the initial positions $X^{i}_0, i \geq 1,$ and of $ (\bar \pi^i)_{i \geq 1} , i \geq 1, $
such that for any $ 0  \le t_1 < t_2 < t , $
\begin{eqnarray}\label{eq:eulerpartcont}
&&M_{t_2}^N - M_{t_1}^N =  \int_{ t_1}^{t_2}a' ( \xnone_{  s }) \sqrt{ \frac1N  \bar f ( X^{N }_{  s }) } d W^{N  }_s + \frac12 \int_{t_1}^{t_2} a'' (\xnone_{  s }) \left(  \frac1N \bar f ( X^{N }_{  s }) \right) ds \nonumber  \\
&& \quad \quad   \quad \quad  \quad \quad  \quad \quad \quad \quad   \quad \quad  \quad \quad  \quad \quad  \quad \quad   \quad \quad  \quad \quad  \quad \quad \quad   \quad \quad  \quad \quad  \quad \quad   +\cR_{t_1, t_2}^{N}, 
\end{eqnarray}
where $\cR^N_{t_1, t_2}$ is an error term satisfying
$$\E |\cR_{t_1, t_2}^{ N  }| \leq C_t (|t_2- t_1| + \sqrt{|t_2 -t_1|})   (\ln N)^{1/5} N^{ - 1/10}.$$
\end{thm}

\section{Proof of Theorem~\ref{th:auxiliary}}
In what follows, we decompose time into small intervals of length $ \delta $ where $\delta = \delta (N) < 1$ will be chosen later. Our strategy is discrete, that is, we first approximate the continuous time process by its discrete time skeleton, sampled at time multiples of $ \delta .$ And then we use each such interval to couple the sum of the small jumps felt by a given neuron with the increment of the Ito integral in the limit process.

The results stated this section hold for any fixed $ \delta  < 1 ,$ up to Subsection \ref{sec:35}. It is only at the very end, within the proof of Theorem \ref{th:auxiliary}, that we have to choose $\delta = \delta( N) :=  (\ln N)^{4/5} N^{ - 2/5}.$

\subsection{Time-Discretization}
For $k\geq 0,$ $N\in\N^*,$ denote
\begin{equation}\label{eq:Stndelta}
M_{\delta}^{N,k}:=\sum_{j=1}^N\int_{(k\delta,(k+1)\delta]  \times \R_+ \times \R }(a(X^{N,1}_{k\delta}+u/\sqrt N)-a(X^{N,1}_{k\delta})) \indiq_{\{ z \le f ( X^{N, j }_{ k\delta}) \}} \pi^{ j } (ds, dz,du).
\end{equation}

The following proposition summarizes the error due to the time discretization and to the symmetrization (adding the term $j=1$ to the sum over $j\neq 1$) of $M_t^N.$
\begin{prop}\label{prop:discretis} For all $t\geq 0,$ $N\in\N^*$ and $ 0 < \delta < 1,$ 
$$\E| M_t^{N}- \sum_{k=0}^{[t/\delta]- 1}M_{\delta}^{N,k}|\leq C_t (t+\sqrt t)(\delta^{1/4}+1/\sqrt N) .$$
\end{prop}
The proof of Proposition \ref{prop:discretis} is given in Section \ref{sec:auxiliary}.

\subsection{Representation by means of a compound Poisson process}
In what follows our goal is to couple each $ M_\delta^{N, k } $ with a Gaussian random variable. 
To construct this coupling we consider the family $(\bar \pi^{j } ( ds, dz ))_{j  \geq 1 } $ of i.i.d. Poisson random measures on $ \R_+ \times \R_+ , $ having Lebesgue intensity, independent of the initial positions $X^{i}_0, i \geq 1,$ introduced in \eqref{eq:projpi} above. These measures are used to represent the acceptance and/or rejection of jumps of particles. 

Introduce
\begin{equation}\label{eq: Nkdelta}
 N_\delta^k =  \sum_{i=1}^N \int_{(k\delta, k \delta +\delta ]  \times \R_+  } \indiq_{\{ z \le f ( X^{N, i }_{ k \delta}) \}} \bar \pi^{ i } (ds, dz) .  
 \end{equation}
We have the following alternative representation of $ M_\delta^{N,  k}.$ 
\begin{prop}\label{prop:U}
There exists an i.i.d. sequence $  (U^{k }_l)_{ k \geq 0,   l \geq 1 } $ of random variables distributed according to $ \nu , $ which is independent of the initial positions $X^{i}_0, i \geq 1,$ and of the Poisson random measures $ \bar \pi^{  j } , j\geq 1 ,$ such that 
for all $ k , $ almost surely, 
\begin{equation}\label{eq:decoupling}
 M_{\delta}^{N,k}=\sum_{l=1}^{N_\delta^k }  [ a( \xnone_{k\delta} + U_l^k / \sqrt{N}) - a( \xnone_{k\delta}) ] .
\end{equation}
Moreover, for fixed $k\geq 0,$ $  (U^{k }_l)_{   l \geq 1 } $ are independent of $\F_{k\delta},$ where $\F_{k\delta}$ has been introduced in \eqref{eq:filtration} above.
\end{prop}

Notice that in \eqref{eq:decoupling} we have decoupled the noise of the acceptance/rejection scheme from the random synaptic weights. The proof of  Proposition \ref{prop:U} is given in Section \ref{sec:auxiliary}.

We continue further developing $M^{N, k }_\delta.$ We have the following representation.

\begin{prop}\label{prop:develSkdelta} 
\begin{equation}\label{eq:snkdelta}
M^{N, k }_\delta  = a' ( \xnone_{k\delta}  ) \frac{1}{\sqrt{N}}  \sum_{l=1}^{N_\delta^k }  U_l^k + \frac12 a'' ( \xnone_{k \delta}) \left( \frac1N \bar f( X^N_{k \delta} )\right)  \delta + R^{N, 1, k}_\delta + R^{N, 2, k }_\delta ,
\end{equation}
where 
\begin{equation}\label{eq:error2}
  R^{N,1, k}_\delta = \frac12 a'' ( \xnone_{k \delta}) \frac{1}{N} \left( \sum_{l=1}^{N_\delta^k} (U_l^k)^2 -  \bar  f( X_{k \delta} ) \delta \right)  
\end{equation}  
and  
\begin{equation}\label{eq:error3} 
R^{N, 2, k }_\delta = \frac{1}{6 N^{3/2}} \sum_{l=1}^{N_\delta^k}  a''' ( \xi (U_l^k)  )(U_l^k)^3    ,
\end{equation}  
with  $ \xi (U_l^k)  \in  [ \xnone_{k\delta} - | U_l^k|/\sqrt{N} , \xnone_{k\delta} + |U_l^k|/\sqrt{N}] .$ 
\end{prop}

\begin{proof}
For any fixed $U $ we have that  
$$ a( \xnone_{k\delta} + U / \sqrt{N}) - a( \xnone_{k\delta}) = a' (\xnone_{k\delta} ) \frac{U}{\sqrt{N}} + \frac12 a'' ( \xnone_{k\delta}) \frac{U^2 }{N} + \frac16 \frac{ a''' (  \xi (U) ) U^3}{N^{3/2}},$$
where the intermediate value $ \xi (U) $ is taken such that  $ \xi (U)  \in  [  \xnone_{k\delta} - |U|/\sqrt{N},  \xnone_{k\delta} + |U|/\sqrt{N}] .$ Replacing $U$ with $U^k_l$ and summing over  $l=1,\ldots, N_{\delta}^k$ gives the result.
\end{proof}

The main observation is now that the term $ \sum_{l=1}^{N_\delta^k }  U_l^k$ arising in \eqref{eq:snkdelta} is equal to a time changed random walk $ S^k \circ N_\delta^k , $ where 
\begin{equation}\label{eq:snk}
 S_n^k = \sum_{l=1}^n U_l^k , n \geq 0.
\end{equation} 
It is precisely this random walk that we couple with a Brownian motion by making use of the famous KMT coupling introduced by  \cite{komlos_approximation_1976}.

\subsection{KMT couplings}
Let $S = (S_n)_{ n \geq 0 } $  be the random walk defined by $  S_n = U_1 + \ldots + U_n, $ with $ (U_n)_{n \geq 1 }$ i.i.d., $U_n  \sim \nu $ (where $\nu$ is as in Assumption \ref{control}, satisfying also Assumption \ref{ass:exp}). Let $B $ be a standard one-dimensional Brownian motion.
Following the terminology proposed by \cite{adrien}, and relying on  Corollary~7.5.5 of \cite{ethier_markov_2005} which summarizes \cite{komlos_approximation_1976}, we say that $ (S, B) $ is a {\it KMT coupling} if $S$ and $B$ are
 constructed on the same probability space,  such that 
\begin{equation}\label{eq:KMT}
 \sup_{n \geq 0} \frac{ |S_n - B_n|}{ \ln (n \vee 2) } \le E < \infty 
\end{equation} 
almost surely, where $E$ is a random variable having exponential moments.\footnote{The fact that $E$ admits exponential moments follows from the proof of Corollary 7.5.5 of \cite{ethier_markov_2005}.}

In what follows we use a classical coupling result. It is based on the fact that it is always possible to realize a probability law on a product space $ \Omega_1 \times \Omega_2 , $ where $ \Omega_1, \Omega_2 $ are Polish, by means of a measurable function $ G : \Omega_1 \times [0, 1] \to \Omega_2.$ More precisely, it suffices to first simulate a random variable $ X_1$ according to the first marginal law and then, independently of this, a uniform random variable $ V, $ uniformly distributed on $ [0, 1].$ The couple $ (X_1, G( X_1, V) ) $ then follows the prescribed joint distribution. Despite the fact that this result seems to be common folklore we did not find a reference within one of the classical text books. For the same reason, \cite{adrien} has summarized this result in his Lemma 3.12 from where we quote the result. Applying this result to the joint distribution of the random walk and the Brownian motion which is defined by the KMT coupling, we obtain the following
\begin{lem}\label{lem:KMT}
There exists a measurable function $ G :  \r^{ \N}  \times [0, 1] \to C ( \r_+ , \r) $ such that, if $ (S_n)_n$ is  a centered random walk with jump law $\nu, $ satisfying Assumptions  \ref{control} and \ref{ass:exp}, and if $ V$ is uniformly distributed on $ [0, 1 ], $ independent of $S,$ then $ (S, G(S, V) ) $ is a KMT coupling.
\end{lem}

We will use the above coupling on  each interval $ (k \delta, (k+1) \delta].$ To do so, we shall also use an i.i.d. family $(V^k)_{k\geq 0 }$ of uniform random variables, uniformly distributed on $ [0, 1 ], $ independent of anything else.

\begin{prop}\label{prop:KtN}
Let $N^k_{\delta}$ be given by \eqref{eq: Nkdelta}, $(U_l^k)_{ l \geq 1 }$ as in Proposition \ref{prop:U} and $ (S_n^k)_{ n \geq 0 } $ as in \eqref{eq:snk}. Then 
$$ B^k := G ( S^k, V^k ) $$
is a Brownian motion which is independent of $(\bar \pi^i), \  i\geq 1,$ and of $ {\mathcal F}_{k \delta}, $ and 
 \begin{equation*}
S^k_{ N_\delta^k } =  \sum_{l=1}^{N_\delta^k }  U_l^k =B^{k}_{N_\delta^k } +   K_\delta^{k}, 
\end{equation*} 
where the random variable $K_\delta^{k}$ satisfies
$$\esp{ | K_\delta^k |} \le C \ln ( N \|f\|_\infty \delta) .$$

\end{prop}
\begin{proof}
By construction, the random walk $(S_n^k)_{ n \geq 0} $ is independent of $(\bar \pi^i), \  i\geq 1,$ and of $ {\mathcal F}_{k \delta}. $ 
According to Lemma \ref{lem:KMT},
$ B^k = G( S^k, V^k ) $ 
achieves the KMT-coupling of the random walk $S^k $ with its approximating Brownian motion $B^k .$ Moreover, by construction, $ B^k $ is also independent of $(\bar \pi^i), \  i\geq 1,$ and of $ {\mathcal F}_{k \delta} .$

In this way, we may rewrite the contribution of the small jumps as
\begin{equation}
S^k_{ N_\delta^k }= \sum_{l=1}^{N_\delta^k }  U_l^k  =   B^{k}_{N_\delta^k } +   K_\delta^{k}, \mbox{ with an error term }  K_\delta^k =  S^{k}_{ N_\delta^k } - B^{k}_{N_\delta^k } 
\end{equation} 
which satisfies 
$$ | K_\delta^k | \le  [\ln (N_\delta^k \vee 2)]  E_k ,$$
with $ E_k  = \sup_{ n \geq 0} \frac{ | S_n^k - B_n^k |}{ \ln ( n \vee 2 ) } $ the error bound of the KMT-coupling. This error term depends only on $S^k$ and $V^k$ and hence is independent of  
$ {\mathcal F}_{ k \delta}$ and in particular of $ N^k_\delta.$ 

Using this independence and Jensen's inequality, we conclude that  
$$ \esp{  | K_\delta^k |}  \le \esp{ \ln ( N_\delta^k \vee 2 ) } \esp{| E_k |} \le \ln \left( \esp{ N_\delta^k } + 2  \right) \esp{|E_k |} ,$$
implying the assertion, since $E_k$  has an exponential moment and 
\begin{equation}\label{eq:moyenne}
\E[N_{\delta}^k]=\E[
 \esp{N_\delta^k | {\mathcal F}_{k \delta} }] = \delta\E[\bar f ( X^N_{k \delta} ) ]\leq N\|f\|_{\infty}\delta.
\end{equation}
\end{proof}

In what follows we propose another representation of the time changed Brownian motion $ B^k_{N_\delta^k }.$

\begin{prop}\label{prop:Wkt}
Define 
\begin{equation}\label{eq: Wkt}
W^{N, k}_\delta :=  \sqrt{\delta/ N_\delta^k } B^k_{N_\delta^k }\, \indiq_{\{N^k_\delta\neq 0\}} + B^k_\delta \, \indiq_{\{N^k_\delta =  0\}}.
 \end{equation}
Then $W^{N, k}_\delta \sim {\cal N} (0, \delta) $ is independent of $(\bar \pi^{i}), i\geq 1, $ and of $ {\mathcal F}_{k \delta}.$ 
\end{prop}

\begin{proof}
Since the Brownian motion $(B^k_t)_t$ is independent of $\bar\pi^i,\, i\geq 1$, and of $ {\mathcal F}_{k \delta},$ 
 we only have to prove that $W^{N, k}_\delta$ has the right law and is independent of $N^k_\delta.$ For that sake let $\phi$ be any bounded real valued Borel measurable test function. Since $ B^k $ is independent of $ N^k_\delta,$ we have 
 $$ \esp{\phi(W^{N, k}_\delta)|N^k_\delta} \indiq_{\{N^k_\delta\neq 0\}}  =\esp{\phi\left (\sqrt{\frac {\delta}{N^k_\delta}}\; B^k_{N_\delta^k}\right )|N^k_\delta } \indiq_{\{N^k_\delta\neq 0\}} =  F( N^k_\delta ) \indiq_{\{N^k_\delta\neq 0\}} , $$
 where we define for any $ n \geq 1, $
 $$F(n) = \esp{\phi\left (\sqrt{\frac {\delta}n}\; B^k_n\right )}.$$
 On the other hand, using the scaling property of the Brownian motion, for any $ n \geq 1, $ 
 $$ F(n) = \esp {\phi(B^k_\delta)}  ,$$
 which does not depend on $n.$ Therefore we may conclude that  
 \begin{equation*}
\esp{\phi(W^{N, k}_\delta)|N^k_\delta} \indiq_{\{N^k_\delta\neq 0\}} =\esp {\phi(B^k_\delta)} \indiq_{\{N^k_\delta\neq 0\}},
\end{equation*}
implying the assertion. 
\end{proof}

Noticing that conditionally on $\F_{k\delta}, $ $N_\delta^k$ is Poisson distributed with parameter  $\bar f ( X^N_{k \delta}) \delta,$ 
it is straightforward to obtain the following final representation of the small jumps.
\begin{prop}\label{prop:final}
$$ \frac{1}{\sqrt{N}}\sum_{l=1}^{N_\delta^k }  U_l^k   =  \sqrt{\frac1N  \bar f ( X^N_{k \delta})} \;  W^{N,k}_\delta  + R^{N,3,   k}_{\delta}, $$
where
$\E |R_{\delta}^{N, 3,  k}|\leq C\ln(N\delta)/\sqrt N.$

\end{prop}
The proof of Proposition \ref{prop:final} is postponed to Section \ref{sec:auxiliary}.

\begin{rem}
Notice that up to now we have only used a coupling of the two compound Poisson variables (conditionally on $ \bar f ( X^N_{k \delta } ) $) $ \sum_{l=1}^{N_\delta^k }  U_l^k $ and $ B^k_{ N_\delta^k },$ and not the coupling with the continuous time Brownian motion. An embedding of the Gaussian sum into a continuous time Brownian motion will only be needed in the next step.
\end{rem}

\subsection{Concatenation}
The above construction is discrete and produces independent Gaussian random variables $ W^{N, k }_\delta \sim {\mathcal N} ( 0, \delta ),$ one for each time step $ (k \delta, (k+1) \delta],$ which are independent of $(\bar \pi^{i}), i\geq 1, $ and of $ {\mathcal F}_{k \delta }.$ We now complete this collection of random variables to a Brownian motion. 

More precisely, in a first step, we complete for each $k \geq 0, $ $W^{N, k }_\delta $ to a piece of Brownian motion by filling in a Brownian bridge $(W^{N, k }_t )_{ 0 \le t \le \delta}  , $ which is independent of $\bar\pi^i,\, i\geq 1$, and of $ {\mathcal F}_{k \delta},$ conditioned on arriving at the terminal value $ W^{N, k }_\delta.$ In  such a way all Brownian motion pieces $ (W^{N, k }_{\cdot} )_{ k \geq 0 }$ attached to different time steps are independent.

In a second step we then paste together all these pieces of Brownian motions by introducing  for any $ t \in [k \delta, (k+1) \delta) $ and any $ k \geq 0, $  
\begin{equation}\label{eq:defW}
 W^N_t := \sum_{ l=0}^{k-1} W^{N, l }_\delta + W^{N, k }_{ t - k\delta}, \; \mbox{ where we put }  \sum_{l=0}^{- 1 } := 0.
 \end{equation}
Notice that in this way, the $W^{N, l }_\delta , l \geq 0, $ serve as increments for this newly defined Brownian motion $ W^N.$ 

\begin{prop}
$(W_t^N)_{t \geq 0} $ is a standard Brownian motion with respect to its own filtration which is independent of $ \bar \pi^i, i \geq 1 ,$ and of $ X^{i}_0, i \geq 1.$ 
\end{prop}

\begin{proof}
This follows immediately from the above construction
by successive conditionings.
\end{proof}

\subsection{Proof of Theorem \ref{th:auxiliary}}\label{sec:35}
We give the proof only for $t_1 = 0; $ the case $ t_1 > 0 $ is treated analogously, decomposing the interval $ (t_1, t_2) $ into intervals of length $\delta.$ 

For $s\geq 0$ define $\tau(s)=k\delta$ for $k\delta\leq s<(k+1)\delta.$
Using Propositions \ref{prop:discretis},  \ref{prop:develSkdelta}, \ref{prop:final} together with the definition \eqref{eq:defW} we can write
\begin{eqnarray}\label{eq:euler}
M_t^N&=&  \int_{ 0}^{\tau(t)}a' ( \xnone_{\tau ( s) }) \sqrt{ \frac1N  \bar f ( X^{N }_{\tau ( s) }) } d W^{N  }_s 
+\frac12 \int_0^{\tau(t)} a'' (\xnone_{\tau ( s) }) \left(  \frac1N \bar f ( X^{N }_{\tau ( s) }) \right) ds \\
&&+O((t+\sqrt t)(\delta^{1/4}+1/\sqrt N))
+\sum_{k=0}^{ [t/\delta]- 1}\sum_{i=1}^3 R_{\delta}^{ N,i, k} ,\nonumber
\end{eqnarray}
where, recalling Proposition \ref{prop:discretis}, $ O((t+\sqrt t)(\delta^{1/4}+1/\sqrt N))$ designs a random variable having $L^1-$norm bounded by $ C_t(t+\sqrt t)(\delta^{1/4}+1/\sqrt N)).$

{\bf Step 1.} Using Proposition  \ref{prop:final} we already know 
\begin{equation}\label{eq:rtotal}
\E \sum_{k=0}^{[\frac t{\delta}]-1} |R_{\delta}^{ N,3,  k}| \leq  C [t/\delta] \ln N/\sqrt N,
\end{equation}
where we have used that $ \ln ( N \delta ) \le \ln N , $ since $ \delta \le 1.$ 

{\bf Step 2.}  We now discuss the control of the error term $\sum_{k=0}^{[\frac t{\delta}]-1}( R_{\delta}^{ N,1, k}+R_{\delta}^{ N,2, k}).$
In what follows we will freely switch between the representation of the finite particle system $ X^N$ by means of the original Poisson random measures $ \pi^i ( ds, dz, du ) $ used in \eqref{eq:dyn} and the alternative representation by means of the centered random variables $ (U_l^k)_{l, k} $ and the Poisson random measures $ \bar \pi^{ i } $ of \eqref{eq:projpi}.
We have
$$
\sum_{k=0}^{[\frac t{\delta}]-1} R_{\delta}^{ N,1, k} = \frac{1}{2 N} \int_{ [0, \tau(t) ] \times \R_+ \times \R}  a'' ( X^{N, 1 }_{\tau ( s-) })  \sum_{j=1}^N u^2 \indiq_{\{ z \le f ( X^{N, j }_{\tau(s-) } )  \}} \tilde \pi^j ( ds, dz, du ) ,$$
where  $ \tilde \pi^j (ds, dz, du ) = \pi^j (ds, dz, du )  - ds dz \nu ( du ) $ are the centered Poisson random measures. Using the independence of $\tilde \pi^j,\; j=1,2\ldots, N,$ together with the fact that these measures are compensated, the $L^2-$isometry for stochastic integrals with respect to compensated Poisson random measures,  the fact that $ U \sim \nu $ possesses a finite fourth moment (since it possesses exponential moments), the boundedness of $a''$ and $f,$ we can write
\begin{eqnarray*}\label{eq:barr2control}
&&\E |\sum_{k=0}^{[\frac t{\delta}]-1} R_{\delta}^{ N,1, k} |^2\leq \frac 1{N^2} \sum_{j=1}^N\E \left (\int_{ [0, \tau(t) ] \times \R_+ \times \R}  a'' ( X^{N, 1 }_{\tau ( s-) })  u^2 \indiq_{\{ z \le f ( X^{N, j }_{\tau(s-) } )  \}}\tilde \pi^j ( ds, dz, du ) \right)^{2} \\
&& \quad  \quad \quad  \quad \quad \leq \frac 1{N^2} \sum_{j=1}^N\E \int_{ [0, t ] \times \R_+ \times \R}  (a'' ( X^{N, 1 }_{\tau ( s-) }))^2   u^4 \indiq_{\{ z \le f ( X^{N, j }_{\tau(s-) } )  \}  )}ds dz d\nu (u) \leq C { t/N}.
\end{eqnarray*}
As a consequence, using the Cauchy-Schwarz inequality,
\begin{equation}\label{eq:barr2control}  
 \E ( |  \sum_{k=0}^{[\frac t{\delta}]-1} R_{\delta}^{ N,1, k} |)\le \left (\E ( |  \sum_{k=0}^{[\frac t{\delta}]-1}R_{\delta}^{ N,1, k} |)^2\right )^{1/2}  \le C \sqrt{ t/N}.
\end{equation}
Moreover we also have the representation 
$$ \sum_{k=0}^{[\frac t{\delta}]-1} R_{\delta}^{ N,2, k}  = \frac{1}{6 N^{3/2}} \int_{ [0, \tau (t) ] \times \R_+ \times \R}  a''' (\xi (u)  )  \sum_{j=1}^N u^3 \indiq_{\{ z \le f ( X^{N, j }_{\tau(s-) } )  \}}  \pi^j ( ds, dz, du ) .$$
Using the exchangeability of the finite neuron system, the boundedness of $a'''$ and of $f$  and the integrability of $|U|^3,$ we get
\begin{equation}\label{eq:barr3control}
  \E ( | \sum_{k=0}^{[\frac t{\delta}]-1} R_{\delta}^{ N,2, k} |) \le C t N^{-1/2}.
\end{equation}
{\bf Step 3.} We now discuss the discretization error. First remark that using the Cauchy-Schwarz inequality and the $L^2-$isometry for the stochastic integral, together with the exchangeability and the  boundedness of $a'$, $a''$and $f$, we immediately get
\begin{eqnarray}\label{eq:smallbout}
&&\esp{\left |\int_{ \tau(t)}^{t}a' ( \xnone_{\tau ( s) }) \sqrt{ \frac1N  \bar f ( X^{N }_{\tau ( s) }) } d W^{N  }_s   + \frac12 \int_{\tau(t)}^{t} a'' (\xnone_{\tau ( s) }) \left(  \frac1N \bar f ( X^{N, j }_{\tau ( s) }) \right) ds        \right |} \nonumber \\
&& \quad \quad   \quad \quad  \quad \quad  \quad \quad \quad \quad   \quad \quad  \quad \quad  \quad \quad \quad \quad   \quad \quad  \quad \quad  \quad \quad \quad \quad   \quad \quad  \quad \quad  \quad \leq C (\sqrt{\delta}+\delta). 
\end{eqnarray}
Finally we deal with the discretization error itself. Using the $L^2-$isometry for stochastic integrals, together with the properties of $a$ and $f,$ we get 
\begin{eqnarray}\label{eq:discretour1}
&&\E \left(  \int_{ 0}^{t}\left (a' ( \xnone_{\tau ( s) }) \sqrt{ \frac1N  \bar f ( X^{N }_{\tau ( s) }) } -
a' ( \xnone_{s }) \sqrt{ \frac1N  \bar f ( X^{N }_{s }) }\right )d W^{N  }_s \right )^2   \nonumber  \\
&&=\E  \int_0^t \left((a' ( \xnone_{\tau ( s) })-a' ( \xnone_{s }) ) \sqrt{ \frac1N  \bar f ( X^{N }_{\tau ( s) }) }+ \right.  \nonumber \\
&&  \quad \quad \quad \quad   \quad \quad  \quad \quad  \quad \quad \quad \quad   \quad \quad  \quad  \left. + 
a' ( \xnone_{s })(\sqrt{ \frac1N  \bar f ( X^{N }_{\tau ( s) }) }- \sqrt{ \frac1N  \bar f ( X^{N }_{s }} ) )\right)^2 ds  \nonumber  \\
&& \leq C  \E \! \int_0^t \!  (a' ( \xnone_{\tau ( s) })-a' ( \xnone_{s }) )^2  \frac1N  \bar f ( X^{N }_{\tau ( s) }) \! +\! 
a' ( \xnone_{s })^2 \!  \frac1N \sum_{i=1}^N | f ( X^{N,i }_{\tau ( s) }) -   f ( X^{N , i }_{s } ) |  ds \nonumber  \\
&& \leq C \E  \int_0^t  |a' ( \xnone_{\tau ( s) })-a' ( \xnone_{s }) |  +
{ \frac1N \sum_{i=1}^N | f ( X^{N,i }_{\tau ( s) }) }-   f ( X^{N, i }_{s } ) |  ds  \nonumber  \\
&&  \quad \quad \quad \quad   \quad \quad \quad \quad    \quad \quad \quad \quad   \leq C \E\int_0^t |\xnone_{\tau ( s) }-\xnone_{s }| ds\leq C_t t\sqrt{\delta}.
\end{eqnarray}
Here to obtain the fourth line, we have used that $ | \sqrt{x} - \sqrt{y} | \le C \sqrt{ |x- y |} ;$ the fifth line follows from the boundedness of $f$ and $a',$ and the last from the Lipschitz continuity of $f$ and $a',$ together with Lemma \ref{lem:euler} stated in Appendix.
Similar arguments imply
\begin{equation}\label{eq:discretour2}
\int_0^{t}\left | a'' (\xnone_{\tau ( s) })   \frac1N \bar f ( X^{N }_{\tau ( s) })- a'' (\xnone_{  s })  \frac1N \bar f ( X^{N }_  s )\right|ds \\
\leq C \int_0^t\E|\xnone_{\tau ( s) }-\xnone_{s }| ds\leq C_t t\sqrt{\delta}.
\end{equation}
Putting together \eqref {eq:rtotal}--\eqref {eq:discretour2} we obtain
\begin{eqnarray*}
 \E |\cR_{0, t}^{ N }| &\leq& C_t ( t + \sqrt{t}) ( \delta^{1/4}+ N^{-1/2} )+[t/\delta]\ln(N)/\sqrt N \\
& \le & C_t ( t + \sqrt{t}) ( \frac{1}{\delta} \ln N/\sqrt N + \delta^{1/4}+ N^{-1/2} ) .
\end{eqnarray*}
Choosing $ \delta = \delta( N) :=  (\ln N)^{4/5} N^{ - 2/5}$ such that the above error terms are balanced, this implies the result.
$\qed$

\section{Proof of Theorem \ref{strongconvergence}}
\subsection{An auxiliary system approaching the limit system}\label{sec:41}
We now use, for fixed $N$ and for $\delta = \delta( N) =  (\ln N)^{4/5} N^{ - 2/5}$ the above constructed Brownian motion $W^N $ together with the collection of Poisson random measures $ \bar \pi^i, i \geq 1,  $ which are independent of $W^N,$ and define a mean field version $(\auxY^{N, 1}, \ldots, \auxY^{N, N}) $ of the limit system $ (\xd^i)_{i \geq 1 }  $ as follows. 
For any $ 1 \le i \le N, $ 
\begin{eqnarray}\label{eq:euleraux}
\auxY^{N, i } _{t } &=& X^{ i }_{0} - \alpha \int_{0}^{t} \auxY^{N, i }_{ s } ds - \int_{[0, t ]   \times \R_+  } \auxY^{N, i }_{s- } \indiq_{\{ z \le f ( \auxY^{N, i}_{ s-}) \}} \bar \pi^i (ds, dz) \\
&&+  \int_{ 0}^{t}\sqrt{ \frac1N  \sum_{j=1}^Nf ( \auxY^{N, j}_{ s }) } d W^{N  }_s   . \nonumber
\end{eqnarray}
This is a classical mean-field approximation of the limit system driven by $ W^N $ and by $ ( \bar \pi^i)_{i \geq 1}.$ We have the following 

\begin{prop}\label{prop:2}
Grant Assumptions~\ref{control}, 
 \ref{ass:2} and \ref{ass:exp}. Then for all $t \geq 0,$ 
$$  \E (\sup_{ s \le t }  |  a(\auxY_{s}^{N,1})-a(X_{s}^{N,1})| ) \le C_t   (\ln N)^{1/5} N^{ - 1/10}.$$ 
\end{prop}

\begin{proof}
We  apply again the bijection $a( \cdot)  ,$ now to $\auxY^N.$ Using Ito's formula, equation \eqref{eq:tobecited} together with \eqref{eq:eulerpartcont}, we obtain 
\begin{eqnarray*}
 &&a(\auxY_{t}^{N,1}) -a(X_{t}^{N,1})= - \alpha \int_0^t ( a' ( \auxY_s^{N, 1} ) \auxY_s^{N, 1} - a' ( X_s^{N, 1}) X_s^{N, 1 } ) ds \\
&&\quad \quad \quad \quad   + \frac12 \int_0^t ( a'' ( \auxY_s^{N ,1 } ) (\frac1N \sum_{j=1}^N f ( \auxY_s^{N, j } ) ) -  a'' ( X_s^{N, 1 } ) (\frac1N \sum_{j=1}^N f ( X_s^{N, j } ) )) ds \nonumber \\
&& + \int_0^{ t}\left( a' (   \auxY_{s}^{N,1}) \sqrt{\frac1N \sum_{j=1}^Nf(\auxY_{s}^{N,j})}- a' ( \xnone_{s} ) \sqrt{ \frac1N  \sum_{j=1}^Nf ( X^{N, j }_{s }} )\right) dW_s^N \nonumber \\
&& + \int_{[0,t]\times\r_+  } [ a(0 ) - a(\auxY_{s-}^{N,1}) ] \indiq_{\{z\leq f\left(\auxY_{s-}^{N,1}\right)\}}-[ a(0 ) - a(X_{s-}^{N,1})] \indiq_{\{z\leq f\left(X_{s-}^{N,1}\right)\}} \bar \pi^1(ds,dz) \nonumber \\
&& \quad \quad \quad \quad   \quad \quad \quad \quad    \quad \quad \quad \quad \quad   \quad \quad \quad \quad   \quad \quad \quad \quad   \quad \quad \quad \quad    \quad \quad \quad \quad  \quad \quad \quad \quad + \cR_{0, t}^{N  }.\nonumber
\end{eqnarray*}  
By \eqref {eq:a} we know that $ | x a' (x) - y a' (y) | + |a'' ( x) - a''(y ) | \le C | a(x) - a( y) | . $ Since $f$ is bounded, we therefore obtain,
 for a suitable constant $C,$ and for any $t \geq 0,$ 
\begin{eqnarray*}
&&\frac{1}{C} \left| a(\auxY_{t}^{N,1}) -a(X_{t}^{N,1}) \right|\leq \\
&&     \int_0^t\left( \left| a(\auxY_{s}^{N,1}) -a( X_{s}^{N,1})  \right| + \frac1N \sum_{j=1}^N \int_0^t  | f(\auxY_{s}^{N,j})- f ( X^{N, j }_{s } )| \right)ds   \\
&&+ \left| \int_{[0,t]\times\r_+  } [ a(0 ) - a(\auxY_{s-}^{N,1}) ] \indiq_{\{z\leq f\left(\auxY_{s-}^{N,1}\right)\}}-[ a(0 ) - a(X_{s-}^{N,1})] \indiq_{\{z\leq f\left(X_{s-}^{N,1}\right)\}} \bar \pi^1(ds,dz) \right| \nonumber \\
&&+\left|\int_0^{ t}\left( a' (   \auxY_{s}^{N,1}) \sqrt{\frac1N \sum_{j=1}^Nf(\auxY_{s}^{N,j})}- a' ( \xnone_{s} ) \sqrt{ \frac1N  \sum_{j=1}^Nf ( X^{N, j }_{s }} )\right) dW_s^N \right|+|\cR_{0, t}^{N  }| ,\nonumber
\end{eqnarray*}
where, up to a multiplicative constant,  $\cR_{0, t}^{ N  }$ is given by Theorem \ref{th:auxiliary}.

Introducing now
$$u_t^N:=\esp{\underset{0\leq s\leq t}{\sup}\left|a(\auxY_{s}^{N,1})-a(X_{s}^{N,1})\right|},$$
we have, using once more the properties \eqref{eq:a} of the function $a(\cdot)$ and the Burkholder-Davis-Gundy inequality,
\begin{eqnarray*}
&&u_t^N\leq Ctu_t^N +
\E \Big[  \int_{[0,t]\times\r_+  }  |a(X_{s-}^{N,1} ) - a(\auxY_{s-}^{N,1})|  \indiq_{\{z\leq f\left(\auxY_{s-}^{N,1}\right)\wedge f\left(X_{s-}^{N,1}\right) \}} + \\
&&| a(0 )\!  -\!  a(X_{s-}^{N,1})| \! \indiq_{\{f\left(\auxY_{s-}^{N,1}\right)<z\leq f\left(X_{s-}^{N,1}\right)\}}
\! + \! |a(0 )\!  -\!  a(\auxY_{s-}^{N,1})| \!  \indiq_{\{f\left(X_{s-}^{N,1}\right)<z\leq f\left(\auxY_{s-}^{N,1}\right)\} } \Big] \! d \bar \pi^1(s,z)\nonumber \\  
&&+\E \left[ \left( \int_0^t    \left(  a' (   \auxY_{s}^{N,1})  \sqrt{\frac{1}{N}\sum_{j=1}^Nf(\auxY_{s} ^{N,j})}-a' ( \xnone_{s} ) \sqrt{\frac{1}{N}\sum_{j=1}^Nf(X_{s} ^{N,j})}   \right)^2ds \right)^{1/2} \right] \nonumber \\
&&+  C_t (t + \sqrt{t})   (\ln N)^{1/5} N^{ - 1/10} \nonumber \\
&&\le  Ctu_t^N + C_t (t + \sqrt{t})   (\ln N)^{1/5} N^{ - 1/10} + \\
&& +
 \E \left[ \left( \int_0^t    \left(  a' (   \auxY_{s}^{N,1})  \sqrt{\frac{1}{N}\sum_{j=1}^Nf(\auxY_{s} ^{N,j})}-a' ( \xnone_{s} ) \sqrt{\frac{1}{N}\sum_{j=1}^Nf(X_{s} ^{N,j})}   \right)^2ds \right)^{1/2} \right] ,
 \end{eqnarray*}
where we have used the properties of $ a (\cdot ) ,$ the fact that $ a$ and $f$ are bounded and the bound on $\E|\cR_{0, t}^{ N }|$  given by Theorem \ref{th:auxiliary}.

We now deal with the stochastic integral term. Observe that $\underbar f=\inf_{x\in \R} f> 0 $ implies that we can use Lipschitz property of $x\to\sqrt x$ on $[\underbar f,+\infty[$ and get
$$ \left| \sqrt{\frac{1}{N}\sum_{j=1}^Nf(\auxY_{s} ^{N,j})} - \sqrt{\frac{1}{N}\sum_{j=1}^Nf(X_{s} ^{N,j})} \right| \le C \frac{1}{N} \sum_{j=1}^N | f(\auxY_{s} ^{N,j}) -  f(X_{s} ^{N,j}) | .$$
Therefore, 
\begin{eqnarray}\label{eq:good}
&& \left(a' (   \auxY_{s}^{N,1})  \sqrt{\frac{1}{N}\sum_{j=1}^Nf(\auxY_{s} ^{N,j})}-a' ( \xnone_{s} ) \sqrt{\frac{1}{N}\sum_{j=1}^Nf(X_{s} ^{N,j})}\right)^2  \\
&& \le C | a' (   \auxY_{s}^{N,1}) - a' ( \xnone_{s} )|^2 + C   \left| \sqrt{\frac{1}{N}\sum_{j=1}^Nf(\auxY_{s} ^{N,j})} - \sqrt{\frac{1}{N}\sum_{j=1}^Nf(X_{s} ^{N,j})} \right|^2 \nonumber \\
&& \le C \left( \sup_{s \le t } | a ( \auxY_{s}^{N,1}) - a ( \xnone_s) |\right)^2 + C \left( \frac1N \sum_{j=1}^N \sup_{s \le t } | a (  \auxY_{s}^{N,j}) - a ( X^{N, j }_s) | \right)^2. \nonumber
\end{eqnarray}
Here we have used the properties of the function $a(\cdot) $ to obtain the last line. By the exchangeability of $ (X^N, \auxY^N) , $ this implies 
\begin{equation}\label{eq:iterate}
 u_t^N \le C(t+\sqrt{t})u_t^N+C_t (t + \sqrt{t})  (\ln N)^{1/5} N^{ - 1/10} .
\end{equation} 
In what follows, $t$ is fixed, serving as our time horizon. We choose $ t_* < t $ sufficiently small such that $C(t_*+\sqrt{t_*}) +  C_{t} (t_* + \sqrt{t_*})   \le \frac12 . $ As a consequence, \eqref{eq:iterate} implies
$$ u_{t_*}^N \le (\ln N)^{1/5} N^{ - 1/10}  .$$

We now wish to iterate this inequality over time intervals $ [kt_*, (k+1) t_*] , k \geq 1, $ such that $ kt_* \le t.$ For that sake let  
$$ \Delta_k := \E ( \sup_{ k t_* \le s \le (k+1) t_* \wedge t}|  a(\auxY_{s}^{N,1})-a(X_{s}^{N,1})| ) .$$
Then we obtain similarly to \eqref{eq:iterate}, using now Theorem \ref{th:auxiliary} on $ [ k t_*, (k+1) t_*],$ 
\begin{eqnarray*}
\Delta_k &\le& u^N_{k t_*} + C(t_*+\sqrt{t_*}) \Delta_k + C_t   (t_* + \sqrt{t_*})  (\ln N)^{1/5} N^{ - 1/10} \\
& \le & u^N_{k t_*} + \frac12 \Delta_k + \frac12  (\ln N)^{1/5} N^{ - 1/10} ,
\end{eqnarray*} 
implying 
$$ \Delta_k \le 2  u^N_{k t_*} +   (\ln N)^{1/5} N^{ - 1/10}   .$$ 
Iterating this inequality we end up with 
$$ \sup \{ \Delta_k \ :  k t_* \le t \}  \le 2^{ \ell+1 }   (\ln N)^{1/5} N^{ - 1/10} ,$$
where $\ell = \max \{ k : k t_* \le t \}=  [ t/t_*].$ 

Therefore,
$$
 \E ( \sup_{ s \le t }  |  a(\auxY_{s}^{N,1})-a(X_{s}^{N,1})| ) \le  \Delta_0 + \ldots + \Delta_{  [t/t_*] } 
 \le (t/t_* +1 ) 2^{ t/t_* +1 }   (\ln N)^{1/5} N^{ - 1/10},
$$
implying the result.
\end{proof}

\subsection{Distance between the auxiliary and the limit system. Proof of Theorem \ref{strongconvergence}.}

Now we control the distance between the auxiliary system and the limit system. For that sake we construct the auxiliary system and the limit system using the same Poisson random measures $ \bar \pi^i , i \geq 1, $ and the same underlying Brownian motion $W^N.$ The conditional independence of the coordinates of the limit system implies

\begin{prop}
\label{coupling}
Grant Assumptions~\ref{control}, \ref{ass:2} and \ref{ass:exp}. Let $ (\bar X^{ N, i })_{ i \geq 1 } $ be the limit system driven by the same Brownian motion $ W^N $ and by the same Poisson random measures $ \bar \pi^i , i \geq 1, $ as $ \auxY^N.$  Suppose that $\auxY^{N,i}_0=\bar X^{N,i}_0= X^i_0$ for all $1 \le i \le N.$ Then for all $  t > 0 , $ 
$$  \E (\sup_{s \le t } | a ( \auxY^{N, 1 }_s) - a ( \bar X^{N, 1}_s) | ) \le C_t N^{-1/2}.$$
\end{prop}

\noindent This result is inequality (4.6) of \cite{ELL_2020}.\\

We conclude with the
\begin{proof}[Proof of Theorem~\ref{strongconvergence}]
The result is now a straightforward consequence of Propositions~\ref{prop:2} and~\ref{coupling}.
\end{proof}

\section{Proofs}\label{sec:auxiliary}
\subsection{Proof of Proposition \ref{prop:discretis}}
\begin{proof}
Using the fact that $a\in\C^3_b(\R, \R_+)$, we immediately get a bound on the symmetrization error.
\begin{eqnarray}\label{eq:symerror}
&&\E \left[  \int_{ [0, t ] \times \R_+ \times \R  }(a(X^{N,1}_{s-}+u/\sqrt N)-a(X^{N,1}_{s-})) \indiq_{\{ z \le f ( X^{N, 1 }_{ s-}) \}} \pi^{ 1 } (ds, dz,du) \right| \\
 && \le \E{ \int_0^t\int_{\R}\left| a' ( X^{N, 1 }_s)  \frac{u}{\sqrt{N}} + a''( X^{N, 1 }_s ) \frac{u^2}{{2 N}}+a^{3} (\xi (u) )  \frac{u^{3}}{6 N\sqrt{N}}\right |ds\nu(du)} \leq \frac {Ct}{\sqrt N } . \nonumber
\end{eqnarray}
We now estimate the error due to time discretization. Denote
 \begin{eqnarray*}
 \bar R_t^{N, \delta}&:=&\sum_{j=1}^N  \int_{[0, t ] \times \R_+ \times \R  } [ a( \xnone_{s-} + u / \sqrt{N}) - a( \xnone_{s-}) ]  \indiq_{\{ z \le f ( X^{N, j }_{ s-}) \}}  \pi^j ( ds, dz, du) -\\
&& \sum_{j=1}^N  \int_{ [0, t]\times \R_+ \times \R  } [ a( \xnone_{\tau(s-) } + u / \sqrt{N}) - a( \xnone_{\tau(s-)}) ]  \indiq_{\{ z \le f ( X^{N, j }_{\tau ( s-)}) \}}  \pi^j ( ds, dz, du) .
\end{eqnarray*}

Using Taylor's formula at order three, we can write 
$\bar R_t^{N, \delta}:=\bar R_t^{N,1, \delta}+\bar R_t^{N,2, \delta}+\bar R_t^{N,3, \delta},$
where $\bar R_t^{N,i, \delta}$ resumes the contribution of the $i$-th derivative, $ i=1,2,3.$
Using the exchangeability of the finite system, the fact that $f$ is bounded together with the properties of the distance function $a(\cdot),$ we obtain
\begin{eqnarray*}
&&  \E ( |  \bar R_t^{N,1,  \delta}|) \le 
  \E \left|  \sum_{j=1}^N  \int_{ [0, t ] \times \R_+ \times \R  }[ a' ( X^{N, 1 }_{s-}) - a' ( \xnone_{\tau(s-) } ] \frac{u}{\sqrt{N}}  \indiq_{\{ z \le f ( X^{N, j }_{ s-}) \}} \pi^j ( ds, dz, du) \right|  \\
&& +    \E \left|  \sum_{j=1}^N  \int_{ [0, t ] \times \R_+ \times \R  }a' ( X^{N, 1 }_{\tau(s-)})  \frac{u}{\sqrt{N}} [  \indiq_{\{ z \le f ( X^{N, j }_{\tau ( s-)}) \}} - \indiq_{\{ z \le f ( X^{N, j }_{ s-}) \}}] \pi^j ( ds, dz, du) \right|. 
\end{eqnarray*} 
To estimate the first expectation, we use the Cauchy-Schwarz inequality, the independence of $\pi^j,\; j=1,2,\ldots , N,$ the $L^2-$isometry for stochastic integrals with respect to compensated Poisson random measures, the fact that $f$ is bounded and $a'$ Lipschitz (recall that $ a''$ is bounded), and the bound  $\E | X^{N, 1}_{\tau ( s) } - X^{N, 1 }_s|  \le C_t \sqrt{\delta},\;   
s \le t $ (see Lemma \ref{lem:euler} below),  to deduce that 
\begin{eqnarray*}
&&\E \left|  \sum_{j=1}^N  \int_{ [0, t ] \times \R_+ \times \R  }[ a' ( X^{N, 1 }_{s-}) - a' ( \xnone_{\tau(s-) }) ] \frac{u}{\sqrt{N}}  \indiq_{\{ z \le f ( X^{N, j }_{ s-}) \}} \pi^j ( ds, dz, du) \right|  \\
 &&\quad \quad \le C \left( \int_0^t  \E | X^{N, 1}_{\tau ( s) } - X^{N, 1 }_s|^2 ds \right)^{1/2}  \le  C_t \sqrt{t \delta}.
\end{eqnarray*}
Similar arguments give
\begin{eqnarray}\label{eq:rn1}
&& \E \left|  \sum_{j=1}^N  \int_{ [0, t ] \times \R_+ \times \R  }a' ( X^{N, 1 }_{\tau(s-)})  \frac{u}{\sqrt{N}} [  \indiq_{\{ z \le f ( X^{N, j }_{\tau ( s-)}) \}} - \indiq_{\{ z \le f ( X^{N, j }_{ s-}) \}}] \pi^j ( ds, dz, du) \right|  \\
&& \le C \left( \int_0^t \frac1N \sum_{j=1}^N  \E | f( X^{N, j}_{\tau ( s) }) - f(X^{N, j }_s) | ds \right)^{1/2}  \le  C_t \sqrt{t } \delta^{1/4} .\nonumber
\end{eqnarray} 
To treat the higher order derivatives, we use again the exchangeability of $(X^{N,j}),\; j=1,\ldots N,$ the boundedness of $f$ and of the third derivative of $a$ and the Lipschitz continuity of $ a''$  to deduce
\begin{equation}\label{eq:rn2}
\E ( |  \bar R_t^{N,2,  \delta}|) \le C_t\sqrt\delta \quad \mbox{ and }  \quad \E ( |  \bar R_t^{N,3,  \delta}|) \le C t/\sqrt N.
\end{equation}
\eqref{eq:rn1} and \eqref{eq:rn2} together with the assumption $\delta <1$ imply that 
$$ \E ( |  \bar R_t^{N, \delta}|)\leq   C_t \sqrt{t } \delta^{1/4} 
+  C_t\sqrt\delta+C t/\sqrt N\leq C_t(\delta^{1/4}+N^{-1/2}),$$
which ends the proof.
\end{proof}

\subsection{Proof of Proposition \ref {prop:U}}
%
Conditioning on $ X^N_{k \delta } = x ,$ 
$$\Pi^k ( du ) := \sum_{j=1}^N \pi^j ( ( k\delta, (k+1) \delta] \times [0, f ( x^j)   ] , du ) $$
is a Poisson random measure on $ \R$ having intensity $ \delta \bar f( x)  \nu ( du ) .$ Its total number of jumps is given by $ N_\delta^k $ introduced in \eqref{eq: Nkdelta}. 

The basic properties of Poisson random measures imply that we have the representation
$$ \Pi^k ( du ) = \sum_{l=1}^{N_\delta^k} \delta_{U_l^k } ( du) , $$
where  the $ ( U_l^{k } )_{l \geq 1 } $ are i.i.d., distributed as $  \nu, $ independent of $ N_\delta^{k } $ and of $ {\mathcal F}_{k \delta}.$ 

Conditioning on $ X^N_{k \delta } = x $ and using Fubini's theorem, we have that 
\begin{eqnarray*}
  M_\delta^{N, k }  &= &\sum_{j=1}^N\int_{(k\delta,(k+1)\delta]  \times \R_+ \times \R }[a(x^1+u/\sqrt N)-a(x^1)]  \indiq_{\{ z \le f ( x^j ) \}} \pi^{ j } (ds, dz,du)\\
  &= &\int_\R  [a(x^1+u/\sqrt N)-a(x^1)]  \left(  \sum_{j=1}^N \int_{ (s, z) \in (k\delta,(k+1)\delta]  \times \R_+}  \indiq_{\{ z \le f ( x^j ) \}} \pi^j ( ds, dz , du ) \right) \\
  &=&  \int_\R  [a(x^1+u/\sqrt N)-a(x^1)]  \Pi^k (du ) = \sum_{l=1}^{N_\delta^k} [ a ( x + U_l^k /\sqrt{N} ) - a(x) ],
\end{eqnarray*}  
and this concludes the proof. 
$\qed$
%
%

\subsection{Proof of Proposition \ref{prop:final}}
\begin{proof}
Denote
$
 \cE_\delta^k :=    \sqrt{N_\delta^k/ \delta} \;  -  \sqrt{ \bar f ( X^N_{k \delta})}.  
$
As a consequence of Proposition \ref{prop:KtN} and Proposition \ref {prop:Wkt}, we immediately  obtain 
$$ \frac{1}{\sqrt{N}}\sum_{l=1}^{N_\delta^k }  U_l^k   =  \sqrt{\frac1N  \bar f ( X^N_{k \delta})} \;  W^{N,k}_\delta  + \frac{1}{\sqrt{N}} \cE_\delta^k\;  W^{N,k}_\delta +  \frac{1}{\sqrt{N}} K_\delta^{k} . $$
Remember that  $W^{N, k}_\delta \sim {\cal N} (0, \delta) $ is independent of $(\bar \pi^{i}), i\geq 1, $ and of $ {\mathcal F}_{k \delta}$. As a consequence,  $ W^{N, k}_\delta   $ and 
$\cE_\delta^k$ are independent and $\E[|W^{N, k}_\delta |]\leq \sqrt \delta.$

Taking $ \delta = \delta (N) = (\ln N)^{4/5} N^{ - 2/5} $ such that $ N \delta (N) \to \infty $ as $ N \to \infty,$ we now show that there exists a suitable constant such that 
\begin{equation}\label{eq: toshow}
 \esp{ | \cE_\delta^k | } \le C.
 \end{equation}
Of course, on the event $\{  N_\delta^k = 0 \}, $ the error term $ \cE_\delta^k $ has no chance to be small. However, conditionally on $ {\mathcal F}_{k \delta}, $  $N_\delta^k \sim Poiss ( \bar f ( X^N_{k \delta} ) \delta ), $ and $ f (\cdot ) \geq \underline f > 0 $ being lowerbounded, 
deviation estimates for Poisson random variables imply that 
\begin{equation}\label{eq:largedev}
\mathbb{P}( N_\delta^k   \le  N \underbar f \delta  / 2 ) \le e^{ - c_1 N  \delta },  
\end{equation}
where $c_1 $ does not depend on $N $ nor on $ \delta$ (see Appendix for a proof). We may therefore introduce the event 
\begin{equation*}
G^k := \{ N_\delta^k / \delta  >  N \underbar f  / 2 \}.
\end{equation*}
We use the Lipschitz continuity of $ x \mapsto \sqrt{x} $ on $[ N\underbar f /2, \infty ), $ with Lipschitz constant 
$$ \frac{1}{ 2 \sqrt{ N \underbar f /2 } } \le C N^{ - 1/2 },
$$
such that for all $ x , y \geq N \underbar f /2,$ 
$$ | \sqrt{x} - \sqrt{y} | \le C N^{-1/2} |x-y|.$$
In what follows, we apply the above inequality, on the event $G^k,$ to $ x= N_\delta^k /\delta $ and to $y = \bar f ( X^N_{ k \delta} ) $. Using  $Var(N^k_{\delta}|\F_{k\delta})=\delta\bar f(X^N_{k\delta})$ this gives 
\begin{eqnarray*}
 \esp{ |\cE_\delta^k | 1_{G^k} }\leq\frac C{\delta\sqrt  N}\esp{| {N^k_{\delta}}-\delta\bar f(X^N_{k\delta})|}
&\leq & \frac C{\delta\sqrt  N}\esp{  (N^k_{\delta}-\delta\bar f(X^N_{k\delta}))^2       }^{1/2}       \\
&  \leq & \frac C{\delta\sqrt N}\sqrt {\delta N \|f\|_{\infty}}=\frac {C \sqrt{ \|f\|_\infty}}{\sqrt\delta}.
\end{eqnarray*}
Moreover, the fact that $f$ is bounded implies that  $ \esp{ |\cE_\delta^k |^2 } \le CN  $ and 
$$ \esp{ (\cE_\delta^k)^2}^{1/2} (\mathbb{P} ( (G^k)^c ))^{1/2} \sqrt{\delta} \le  C \sqrt{N \delta}  e^{ - (c_1/2)  N \delta}.$$ 
Finally, 
$$ \esp{ | \cE_\delta^k |} \le \esp{ | \cE_\delta^k | 1_{G^k} } +  \esp{ | \cE_\delta^k | 1_{(G^k)^c} }\le  \esp{ |\cE_\delta^k | 1_{G^k} } + \esp{ (\cE_\delta^k)^2}^{1/2} (\mathbb{P} ( (G^k)^c ))^{1/2} .$$ 

Since $ N \delta (N) \to \infty $ as $ N \to \infty ,$ 
we have that $ \sup_N  \sqrt{N \delta}  e^{ - (c_1/2)  N \delta} < \infty ,$   which implies the assertion \eqref{eq: toshow}. 
Together with the bound $\E[|K^k_{\delta}|]\leq C\ln(N\delta)$ given by the Proposition \ref{prop:KtN} it completes the proof.
\end{proof}


\begin{appendix}
\section*{}
\subsection{Useful a priori upper bounds and remaining proofs}
Assumption \ref{control} together with the fact that $f$ is bounded implies that the following apriori bounds hold.
\begin{prop}
For all $ t > 0, $ for $X_{s}^{N,1}$ given by \eqref{eq:dyn},
\begin{equation}\label{eq:apriori1}
\underset{N\in\n^*}{\sup}\underset{0\leq s\leq t}{\sup}\esp{\left(X_{s}^{N,1}\right)^2}<+\infty \mbox{ and }  
\underset{N\in\n^*}{\sup}\esp{\underset{0\leq s\leq t}{\sup}\left|X^{N,1}_s\right|}<+\infty.
\end{equation}
\end{prop}
\noindent The above assertion is Lemma 4.1 of \cite{ELL_2020}. \\

\begin{lem}\label{lem:euler}
Suppose $\delta \le 1.$ Then we have for all $ s \le t , $ $ \E ( | \xnone_{\tau( s) } - \xnone_s | ) \le C_t \sqrt{\delta }.$  
\end{lem} 

\begin{proof}
It suffices to perform the proof for $ s < \delta $ and $ \tau (s) = 0.$ 
Then 
\begin{eqnarray*}
\xnone_s - X^1_0 &=& - \alpha \int_0^s \xnone_r dr -   \int_{[0,s]\times\r_+\times\r} X^{N, 1}_{r-}  \indiq_{ \{ z \le  f ( X^{N, 1}_{r-}) \}} \pi^1 (dr,dz,du) \\
&&+  \frac{1}{\sqrt{N}}\sum_{ j \neq 1 } \int_{[0,s]\times\r_+\times\r}u \indiq_{ \{ z \le  f ( X^{N, j}_{r-}) \}} \pi^j (dr,dz,du). 
\end{eqnarray*} 
Thanks to our a priori estimates \eqref{eq:apriori1}, 
$$ \E \left| \int_0^s \xnone_r dr +   \int_{[0,s]\times\r_+\times\r} X^{N, 1}_{r-}  \indiq_{ \{ z \le  f ( X^{N, 1}_{r-}) \}} \pi^i (dr,dz,du) \right| \le C_t \delta.$$
Moreover
\begin{eqnarray*}
&& \E  \left| \frac{1}{\sqrt{N}}\sum_{ j \neq 1 } \int_{[0,s]\times\r_+\times\r}u \indiq_{ \{ z \le  f ( X^{N, j}_{r-}) \}} \pi^j (dr,dz,du)\right|  \\
&& \le \left(  \frac1N \sum_{ j \neq 1 }  \E \left(\int_{[0,s]\times\r_+\times\r}u \indiq_{ \{ z \le  f ( X^{N, j}_{r-}) \}} \pi^j (dr,dz,du)\right)^2 \right)^{1/2} \le \sqrt{\| f \|_\infty } \sqrt{ \delta }  ,
\end{eqnarray*} 
for $s \le \delta,$ since $f$ is bounded and $ \int_\r u^2 \nu ( du ) = 1.$ 

\end{proof}

\noindent We finally give the \\

\noindent{\it Proof of \eqref{eq:largedev}.}
It is well-known that for a standard Poisson process $(Z_t)_{t \geq 0} $ having  rate $1,$ 
$$ \mathbb{P} (   \sup_{ s\le t } | Z_s - s | \geq \varepsilon ) \le 2 \exp ( - t h( \varepsilon / t)) , $$
for all $ \varepsilon < t , $ 
where $ h(x) = (1 +x ) \ln (1+x) - x$ (see e.g. Theorem 2.3.1 of \cite{Bremaudbook}). 

Since conditionally on  
$\{ X^N_{ k \delta} = x\}  , $  $N_\delta^k  \sim Z_{\bar f ( x) \delta}  , $ we apply the above inequality with $ t = \| f\|_\infty N \delta $ and $ \varepsilon = N \underbar f \delta /2 $ and obtain 
\begin{eqnarray*}
 \mathbb{P} ( N^k_\delta \le  N \underbar f \delta /2 | X^N_{ k \delta}  = x   )&=& \mathbb{P} (  Z_{ \bar f (x)  \delta } \le   N \underbar f \delta /2 ) \le \mathbb{P} (  Z_{ \bar f (x)  \delta } \le   \bar f ( x ) \delta /2 )   \\
 &=&  \mathbb{P} ( Z_{ \bar f (x) \delta }- \bar f ( x ) \delta  \le  - \bar f ( x ) \delta /2 ) \le  \mathbb{P} ( \sup_{ s \le t } | Z_s - s | \geq  \underbar f N \delta /2 ) \\
 & \le & 2 \exp ( - t h ( \varepsilon /t )) , 
\end{eqnarray*}
since $N \underbar f \delta \le  \bar f ( x ) \delta \le N \| f \|_\infty \delta.$ Observing that for our choice of $t$ and $ \varepsilon, $ $\varepsilon/ t = \frac{ \underbar f }{ \| f \|_\infty 2 } $ does not depend on $ N, $ nor on $ \delta $ and integrating with respect to the law of $ X^N_{ k \delta } $ implies \eqref{eq:largedev}.
\qed

\subsection{Why the method of Kurtz does not work in a diffusive scaling}\label{sec:appendixkurtz}
The goal of this section is to explain why the approach proposed in \cite{kurtz78}, which is based on time-change and which we have sketched in our Introduction, in Section \ref{sec:kurtz} above, does not apply in the present frame. Consider the auxiliary particle system $ \tilde X^N $ introduced in \eqref{eq:auxnaive} above. In what follows, to stress the main ideas, we suppose w.l.o.g. that the reset terms are absent, that is, we have for each $ 1 \le i \le N, $ 
$$ X_t^{N, i } = X_0^i  - \alpha \int_0^t  X_s^{N , i } ds + \frac{1}{\sqrt{N}} Z_{ A_t^N}$$
and
$$ \tilde X_t^{N, i } = X_0^i  - \alpha \int_0^t \tilde X_s^{N , i } ds + \frac{1}{\sqrt{N}} B_{\tilde A_t^N},$$
and the two processes $Z$ and $ B $ are coupled according to \cite{komlos_approximation_1975, komlos_approximation_1976} such that 
$N^{-1/2} Z_{ A_t^N} =  N^{-1/2} B_{A_t^N} + K_t ,$ where $K_t$ is the error term coming from the coupling.
For simplicity, we work on some fixed time interval $ [0, T ]$ such that $ \alpha T \le \frac12.$  

We start investigating $  \frac{1}{\sqrt{N}}\left| B_{A_t^{N}} -B_{\tilde A_{t}^{N}}\right|  .$ Notice that we have the upper bound 
$$ A_t^{N} \vee \tilde A_t^{N} \le N T \|f \|_\infty.$$
So we introduce the  modulus of continuity of $B$ for this time horizon
$$ M := \sup_{ u, v \le N T  \|f\|_\infty  }  \frac{ |  B_u -  B_v|}{ \varphi ( |u-v|)} ,$$
where 
$$ \varphi ( x) = \sqrt{ x ( 1 + \ln ( \frac{ N T \|f \|_\infty  }{ x}) ) } , \; 0 \le x \le N T \|f\|_\infty .$$ 
Then  there exists some $ \lambda > 0 $ such that $ \esp{ e^{\lambda M} } < \infty $ (see e.g. \cite{kurtz78}, Lemma 3.2). We obtain the upper bound 
$$  \frac{1 }{\sqrt{N}}\left| B_{A_t^{N}} -B_{\tilde A_{t}^{N}}\right| \le  \frac{1 }{\sqrt{N}} M \varphi \left(  | A_t^{N } - \tilde A_t^{N} |\right) .$$
Since $ f$ is Lipschitz, we certainly have for all $ t \le T,$ 
\begin{equation}\label{eq:modcont}
 | A_t^{N } - \tilde A_t^{N} |  \le  \|f\|_{Lip} T   \sum_{i = 1 }^N  \sup_{ s \le T} | X^{N,i}_s -  \tilde X^{N, i }_s | = :   \|f\|_{Lip} T  G_T^N .
\end{equation}
Fix some $ \beta \in (0, 1 )  $ and consider the event 
$$ {\mathcal {G}}_N := \left\{   G_T^N  \geq  N^\beta   \right\} .$$
On $ {\mathcal {G}}^c_N, $ we have by exchangeability
$$ \esp{ \sup_{s \le T } | X^{N, i }_s - \tilde X^{N, i }_s| \indiq_{ {\mathcal {G}}^c_N} } = \frac{1}{N} \esp{ G_T^N \indiq_{ {\mathcal {G}}^c_N} } ) \le N^{ - ( 1 - \beta ) } \to 0 .$$
So it is sufficient to work on $ {\mathcal {G}}_N.$ 
On this set, using that $ \varphi $ is increasing, we have 
$$\varphi \left(  | A_t^{N } - \tilde A_t^{N} |\right)  \le C_T \sqrt{ G_T^N} \sqrt{ 1 + \ln N} ,$$
such that for each $ 1 \le i \le N $ and $ t \le T, $ recalling the upper bound on $|K_t| $ given in \eqref{eq:kt}, 
\begin{eqnarray*} 
 |X_t^{N, i } - \tilde X_t^{N, i } | &\le& \alpha T \sup_{s \le T} | X_s^{N, i } - \tilde X_s^{N, i } | + C_T \frac{1}{\sqrt{N}} M \sqrt{ G_T^N} \sqrt{ 1 + \ln N} + \frac{\ln ( T N \|f\|_\infty )}{\sqrt{N}} E  \\
& \le & \frac12 \sup_{s \le T} | X_s^{N, i } - \tilde X_s^{N, i } | + C_T \frac{1}{\sqrt{N}} M \sqrt{ G_T^N} \sqrt{ 1 + \ln N} + \frac{\ln ( T N \|f\|_\infty )}{\sqrt{N}} E  .
\end{eqnarray*}
Taking the supremum over all $ t \le T $ on the left hand side, subtracting $\frac12 \sup_{s \le T} | X_s^{N, i } - \tilde X_s^{N, i } | $ on both sides and multiplying by $2,$ we thus obtain, after having summed over all $ i , $ 
\begin{equation}\label{eq:notgood}
 G_T^N \le  C_T \sqrt{N} M \sqrt{ G_T^N} \sqrt{ 1 + \ln N} + C_T \sqrt{N} (1 +  \ln N )  E  , 
\end{equation} 
 and this inequality holds on  ${\mathcal {G}}_N.$

The conclusion now follows as in the proof of Theorem 2 in \cite{annaetal}. Indeed, the above inequality is quadratic in $ x:= \sqrt{ G_T^N} $ and can be read as $ p(x) = x^2 + b x + c \le 0 ,$ for suitable $b, c.$ All positive $x$ satisfying $p(x) \le 0 $ are necessarily such that $ x^2 \le b^2 ,$ and this upper bound is sharp for large $N$ (we skip the details),  whence 
$$ G_T^N  \le C_TN M^2 ( 1 + \ln N) ,$$
which holds on  ${\mathcal {G}}_N.$ 

Once more, by exchangeability, 
$$ \esp{ \indiq_{\mathcal {G}_N} \sup_{ s \le T } | X^{N, i }_s - \tilde X_s^{N, i } |} = \frac{1}{N} \esp{ \indiq_{\mathcal {G}_N} G_T^N }  \le C_T (1 + \ln N) \esp{M^2 },$$
which does not vanish as $ N \to \infty .$ 

Notice that in \cite{kurtz78} who works in the scaling of the law of large numbers, in \eqref{eq:notgood} above, the term $ \sqrt{N}$ on the right hand side disappears, and this is why this approach works in this case -- while in the present framework of a diffusive scaling the above control does not allow to conclude. 
\end{appendix}

\begin{acks}[Acknowledgments]
The authors thank the anonymous referees for the positive feedback, many useful remarks and careful reading. E.L. thanks Vlad Bally for discussions about Euler schemes and the non-need of approximating the particle system by its Euler scheme, and Nicolas Fournier for pointing out an error in an early version of this manuscript. This work has been conducted as part of  the  FAPESP project Research, Innovation and Dissemination Center for Neuromathematics(grant 2013/07699-0) and of the ANR project ANR-19-CE40-0024.
\end{acks}

\bibliography{biblio}
\end{document}